\documentclass[a4paper,11pt]{amsart}

\usepackage[top=3.3cm, bottom=3cm, left=2.5cm, right=2.5cm]{geometry}
\usepackage{latexsym,amsmath,amscd,amsthm,verbatim,ifthen,amssymb,mathtools,booktabs}

\usepackage{hyperref}
\usepackage{caption}
\usepackage{graphicx, adjustbox, multirow}

\usepackage{enumerate}
\usepackage{palatino}
\usepackage{mathpazo}

\usepackage[all]{xy}
\usepackage{tikz}
\usetikzlibrary{positioning,arrows,cd}

\makeatletter
\def\l@subsection{\@tocline{2}{0pt}{2.5pc}{5pc}{}}
\makeatother

\newcommand{\id}{{\rm id}}
\newcommand{\TC}{{\sf TC}}
\newcommand{\tc}{{\sf TC}}
\newcommand{\cd}{{\rm cd}}
\newcommand{\cat}{{\sf cat}}
\newcommand{\secat}{{\sf secat}}

\newcommand{\height}{{\rm height}}

\newcommand{\E}{{\rm E}}
\newcommand{\D}{{\mathcal D}}

\newcommand{\quot}[2]{\left.\raisebox{.2em}{$#1$}\middle/\raisebox{-.2em}{$#2$}\right.}

\renewcommand{\setminus}{\smallsetminus}

\newcommand{\ZG}{\mathbb{Z}[G]}
\newcommand{\ZH}{\mathbb{Z}[H]}
\newcommand{\ZC}{\mathbb{Z}[C]}

\newcommand{\zgh}{\Z[G/H]}

\newcommand{\NN}{\mathbb{N}}
\newcommand{\ZZ}{\mathbb{Z}}

\newcommand{\Z}{\mathbb{Z}}

\newcommand{\C}{\mathcal{C}}

\newcommand{\im}{\mathrm{im}\,}

\newcommand{\Res}{\mathrm{Res}}
\newcommand*{\longhookrightarrow}{\ensuremath{\lhook\joinrel\relbar\joinrel\rightarrow}}

\DeclareMathOperator{\Hom}{Hom}
\DeclareMathOperator{\Ext}{Ext}

\DeclareMathOperator{\Img}{im}

\newcommand{\res}{\mathrm{Res}}
\newcommand{\Ind}{\mathrm{Ind}}
\newcommand{\vv}{\mathfrak{v}}
\newcommand{\Aut}{\mathrm{Aut}}

\newcommand{\ev}{\mathrm{ev}}

\newtheorem{theorem}{Theorem}[section]

\newtheorem{proposition}[theorem]{Proposition}
\newtheorem{lemma}[theorem]{Lemma}
\newtheorem{corollary}[theorem]{Corollary}

\newtheorem*{thmi}{Theorem}

\theoremstyle{definition}
\newtheorem{definition}[theorem]{Definition}

\theoremstyle{remark}
\newtheorem{example}[theorem]{Example}					
\newtheorem{remark}[theorem]{Remark}
\newtheorem{notation}[theorem]{Notation}

\numberwithin{equation}{section}

\begin{document}


\title[Sequential TC of aspherical spaces and sectional categories of subgroup inclusions]{Sequential topological complexity of aspherical spaces and sectional categories of subgroup inclusions}


\author[A. Espinosa Baro]{Arturo Espinosa Baro}
\address{Faculty of Mathematics and Computer Science,
	Adam Mickiewicz University, Umultowska 87, 60-479 Pozna\'n, Poland}
\email{arturo.espinosabaro@gmail.com, artesp1@amu.edu.pl}

\author[M. Farber]{Michael Farber}
\address{School of Mathematical Sciences \\
Queen Mary, University of London\\
London, E1 4NS\\
United Kingdom}
\email{m.farber@qmul.ac.uk}
\author[S. Mescher]{Stephan Mescher}
\address{Institut f\"ur Mathematik \\ Martin-Luther-Universit\"at Halle-Wittenberg \\ Theodor-Lieser-Strasse 5 \\ 06120 Halle (Saale) \\ Germany}
\email{stephan.mescher@mathematik.uni-halle.de}
\author[J. Oprea]{John Oprea}
\address{Department of Mathematics\\
Cleveland State University\\
Cleveland OH 44115  \\
U.S.A.}
\email{jfoprea@gmail.com}

\makeatletter
\@namedef{subjclassname@2020}{\textup{2020} Mathematics Subject Classification}
\makeatother


\begin{abstract}
We generalize results from topological robotics on the topological complexity (TC) of aspherical spaces to sectional categories of fibrations inducing subgroup inclusions on the level of fundamental groups. In doing so, we establish new lower bounds on sequential TCs of aspherical spaces as well as the parametrized TC of epimorphisms. Moreover, we generalize the Costa-Farber canonical class for TC to classes for sequential TCs and explore their properties. We combine them with the results on sequential TCs of aspherical spaces to obtain results on spaces that are not necessarily aspherical.
\end{abstract}

\maketitle

\setcounter{tocdepth}{1}

\tableofcontents

\section*{Introduction}

An autonomously functioning system in robotics operates a motion
planning algorithm which takes as input the initial and the desired
states of the system and produces as output a motion of the system from the initial to the desired
state. A topological approach to the robot motion planning problem was initiated in
\cite{Farber03}; the topological techniques explained relationships between instabilities
occurring in robot motion planning algorithms and topological features of robots'
configuration spaces. In this article we focus on the special case when the configuration space of the system is aspherical and analyse the arising topological and algebraic problems.

\subsection*{Sectional category and topological complexity} The \emph{sectional category} of a fibration $p:E \to B$ is given as the minimal value of $n$, for which  $B$ admits an open cover consisting of $n+1$ non-empty open sets, each with the property that $p$ admits a continuous local section over it. It was introduced under the name of \emph{genus} of a fibration by A. Schwarz in \cite{Schwarz66} and is usually denoted by $\secat(p:E \to B)$, or simply by $\secat(p)$. 

While sectional category can be seen as a generalization of the Lusternik-Schnirelmann category of a space, there is another special case of sectional category around whom a huge circle of ideas has been established in the past 20 years: the notion of topological complexity (TC). It has been introduced by M. Farber in \cite{Farber03}, see also \cite[Chapter 4]{Farber06b} or \cite{Farber08} for introductions to the subject and the introduction to \cite{FGLO17} for an overview of more recent developments. The TC of a path-connected topological space $X$ is defined as $\TC(X)=\secat(p)$, where $p$ denotes the path fibration
$$p: PX \to X \times X, \quad p(\gamma) = (\gamma(0),\gamma(1)),$$
with $PX=C^0([0,1],X)$. It is motivated by the field of topological robotics and it encodes the complexity of motion planning for autonomous robots in an arbitrary work space. 

An extension of TC is given by the notion of  \emph{sequential} or \emph{higher topological complexity}, which has been introduced by Y. Rudyak in \cite{RudyakHigher} and has been explored by various authors ever since, see e.g. \cite{BGRT}, \cite{GonzGrant} or \cite{GGY}. The $r$-th sequential topological complexity of a path-connected space $X$, where $r \geq 2$, is defined as $\TC_r(X)=\secat(p_r)$, where 
\begin{equation}
\label{EqIntrSeqTC}
p_r: PX \to X^r, \qquad p_r(\gamma) = (\gamma(0),\gamma(\tfrac{1}{r-1}),\dots,\gamma(\tfrac{r-2}{r-1}),\gamma(1)). 
\end{equation}
Obviously, it holds that $p_2=p$, the path fibration of $X$. It encodes a variation of the motion planning problem from topological robotics in which the robot is supposed to visit several pre-determined intermediate points along its way from start to end point.

\subsection*{Topological complexity of aspherical spaces} Both TC and its sequential version are homotopy invariants. Thus, the TC of an \emph{aspherical}, or Eilenberg-MacLane space, depends only on its fundamental group. It is a natural question to ask for a description of the TC of an aspherical space as an algebraic expression of its fundamental group. For the  Lusternik-Schnirelmann category of an aspherical space, such a description is given by the celebrated Eilenberg-Ganea theorem, see \cite{EG65}. 

The TC of aspherical spaces has received a lot of attention and has been studied from various perspectives. Complementary to computations for various classes of examples, see among others \cite{CohenPruidze}, \cite{GRM},  \cite{DranishNonor} and \cite{CVKlein}, there are also more general approaches to the problem. In \cite{FGLO17} methods and constructions from equivariant topology and Bredon cohomology are employed to produce new lower and upper bounds for the TC of aspherical spaces, while in \cite{FarMes20} the authors establish a spectral sequence from which a lower bound on the TC of an aspherical space can be derived.  

So far, less is known about the sequential TC of aspherical spaces. In \cite{FOSequ} the Bredon cohomology approach from \cite{FGLO17} is transferred to sequential TC yielding lower and upper bounds for sequential TC as well. There are also computations of sequential TCs of certain classes of aspherical spaces in the literature, see e.g. \cite{GGGHMR} for the case of a closed oriented surface. 

\subsection*{Sectional categories of fibrations inducing subgroup inclusions} As a more general question, one might investigate the sectional category of fibrations between aspherical spaces. It is well-known that the fibre homotopy type of such a fibration only depends on the homomorphism it induces on the level of fundamental groups. Using established properties of sectional category, this shows that the sectional category of a fibration between aspherical spaces depends only on its associated homomorphism of fundamental groups. Thus, one might expect a description of such a sectional category as an algebraic expression of its associated group homomorphism. 

 Since $PX$ has the homotopy type of $X$, the path fibration $p$ can be seen as an example of such a fibration in the case that $X$ is aspherical. One checks that if $X$ is of type $K(\pi,1)$ for some group $\pi$, then on fundamental group level, $p$ induces the diagonal inclusion 
 $$\pi \hookrightarrow \pi \times \pi, \quad g \mapsto (g,g).$$
More generally, for any $r \geq 2$, the fibration $p_r$ from \eqref{EqIntrSeqTC} induces the inclusion
$$\pi \hookrightarrow \pi^r, \quad g \mapsto (g,g,\dots,g),$$
between fundamental groups in the aspherical case. 

Motivated by these observations, the authors of  \cite{BCE22} study the sectional category of fibrations between aspherical spaces that induce subgroup inclusions between fundamental groups in loc. cit. in a systematic way. Their main algebraic tools are algebraically defined relative versions of Berstein-Schwarz classes that can be used to provide lower bounds on the corresponding sectional categories. 
As we shall see in the course of this article, studying sectional categories of fibrations inducing subgroup inclusions further subsumes the parametrized topological complexity of epimorphisms. The latter has been studied by M. Grant \cite{Grant22} and can be seen an an analogue of the TC of aspherical spaces in the framework of parametrized TC. This notion has been introduced by D. Cohen, M. Farber and S. Weinberger in \cite{CFW} and has been studied by several authors since then. 

\subsection*{Contents of this article.} The purpose of this article is twofold. On the one hand, we present  generalizations of the spectral sequence constructions from \cite{FarMes20} to sectional categories of fibrations inducing subgroup inclusions. Here, we only make use of methods from homological algebra and obtain a general lower bound for such sectional categories. We also generalize some of the main results of \cite{FarMes20} to sectional categories of subgroup inclusions. While a certain spectral sequence constructed therein can be used to derive lower bounds for the TC of aspherical spaces, we show how to generalize the construction to sectional categories of fibrations inducing subgroup inclusions. Eventually, we derive  the following lower bound for such sectional categories.

\begin{thmi}[Theorem \ref{TheoremSecatKappa}]
	Let $G$ be a geometrically finite group and let $H \leq G$ be a subgroup. Then
	$$\secat(H \hookrightarrow G) \geq \cd(G)-\kappa_{G,H},$$
	where $\kappa_{G,H} = \max \{\cd(H \cap xHx^{-1}) \ |\ x \in G \setminus H\}$.
\end{thmi}
The proof of this lower bound is given by entirely algebraic methods that do not involve any topological arguments at all. 

On the other hand, we carry out the consequences of this general result for sequential TC and approach the question of sequential TC of aspherical spaces from a more topological point of view. First of all, we show that in the case of sequential TC, the above lower bound can be rephrased in terms of cohomological dimensions of centralizers as in part a) of the following theorem.

\begin{thmi}
	Let $\pi$ be a geometrically finite group and let $r \in \NN$ with $r \geq 2$.Let $C(g)$ denote the centralizer of $g \in \pi$ and put $$k(\pi)=\max \{\cd(C(g)) \ | \ g \in \pi\setminus \{1\}\}.$$
	\begin{enumerate}[a)]
		\item (Theorem \ref{TheoremLowerSeqTC})\enskip Then
	$$\TC_r(K(\pi,1))\geq r \cdot \cd(\pi)-k(\pi).$$
	\item (Corollary \ref{CorSeqTClowerNonAsph}) \enskip Let $X$ be a connected locally finite CW complex with $\pi_1(X)=\pi$ whose universal cover is $(k-1)$-connected for some $k \in \NN$ with $\cd(\pi)\leq k$. Then 
$$	\TC_r(X)\geq r \cdot \cd(\pi)-k(\pi).$$
	\end{enumerate}
\end{thmi}
As an application, we will show that for a free amalgamated product of the form $\pi_1 *_H \pi_2$, where $H$ is malnormal both in $\pi_1$ and in $\pi_2$, we obtain that
$$\TC_r(K(\pi_1*_H \pi_2,1)) \geq r \cdot \cd(\pi_1 *_H \pi_2)-\max\{k(\pi_1),k(\pi_2)\}\qquad \forall r \geq 2.$$
We further introduce a generalization of the canonical cohomology class of a topological space, which has been introduced by A. Costa and M. Farber in \cite{Costa} in the study of lower bounds for TC. Our more general notion provides an analogue of this class for the $r$-th sequential TC of a space for any $r \geq 2$. We study its properties and link it to the algebraically defined relative Berstein-Schwarz classes in the case of sequential TC of aspherical spaces. In constrast to relative Berstein-Schwarz classes, canonical classes are defined for all, i.e. not necessarily aspherical topological spaces. Towards the end of the article, we will show how the results for aspherical spaces can be employed to derive results for not necessarily aspherical spaces as well, culminating in a proof of part b) of the above theorem.

\subsection*{Structure of the article} In Section 1 we recall various definitions and results on sectional categories with a focus on subgroup inclusions. Section 2 generalizes the notion of essential cohomology classes as introduced in \cite{FarMes20} for TC to the abovementioned situations involving subgroup inclusions. They are explored purely in terms of group cohomology and further used to compute the sectional category of the inclusion of a \emph{normal} subgroup. The abovementioned spectral sequence from \cite{FarMes20} is then generalized to fibrations inducing subgroup inclusions in Section 3. Its properties are studied in analogy with \cite{FarMes20} and it is shown that the spectral sequence encodes the properties of essential classes of subgroup inclusions. In Section 4 the spectral sequence is used to derive our main lower bound on sectional categories of fibrations inducing subgroup inclusions. These results generalize observations for TC that are implicitly contained, though not explicitly stated, in \cite{FarMes20}. 

We then turn our attention to the special cases of sequential and parametrized topological complexities in Section 5. We state and slightly simplify the corresponding lower bounds and present some simple consequences. Section 6 approaches sequential TC from a more topological point of view. We establish the notion of $r$-th canonical classes that are the sequential TC versions of Costa-Farber's construction for TC, i.e. the case of $r=2$, and show that they coincide with the respective relative Berstein-Schwarz classes in the aspherical case. In the final Section 7 we connect our previous results with the notion of sequential $\D$-topological complexity that was introduced in \cite{FOSequ} and use it to derive results on sequential TCs of spaces that are not necessaily aspherical.

\section*{Acknowledgements}

Michael Farber was supported by the EPSRC research grant EP/V009877/1.

Arturo Espinosa Baro has been supported by Polish National Science Center research grant UMO-2022/45/N/ST1/02814 and by a doctoral scholarship of Adam Mickiewicz University.  

The collaboration between Arturo Espinosa Baro and Stephan Mescher started at the research group \emph{Some problems in applied and computational topology} in February 2023 at Banach Center B\c{e}dlewo, which was organized by Wac\l{}aw Marzantowicz. A.E. and S.M. thank Prof. Marzantowicz and the Banach Center for initiating the collaboration that led to parts of this article. A.E. further thanks Martin-Luther-Universit\"at Halle-Wittenberg for its hospitality at a research visit in which part of this work was carried out.

The authors also wish to express their gratitude to Mark Grant for his many useful commentaries on an earlier version of the manuscript that helped improving the quality of this text.

\section{Preliminaries}
In this brief section, we recall various results and terminology for sectional categories and aspherical spaces.

\begin{definition}
\label{DefSecat}
	Let $X$ and $Y$ be topological spaces and let $f:X \to Y$ be continuous. The \emph{sectional category of $f$}, denoted by $\secat(f:X \to Y)$, or just $\secat(f)$, is given as the minimum number $n \in \NN$ for which there is an open cover $U_0,U_1,\dots,U_n$ of $Y$, such that for each $i \in \{0,1,\dots,n\}$ there exists a continuous map $s_i: U_i \to X$ for which $f \circ s_i$ is homotopic to the inclusion of $U_i$. If there is no finite open cover of $Y$ which has these properties, we put $\secat(f):= + \infty$.
\end{definition}

\begin{remark}
\begin{enumerate}[(1)]
	\item The sectional category of a fibration was originally defined under the name of \emph{genus} by A. Schwarz in \cite{Schwarz66}. Its generalization to arbitrary continuous maps was first considered by A. Fet in \cite{Fet} and by I. Berstein and T. Ganea in \cite{BersteinGanea}, see also \cite{ArkStrom} for a more general approach. If $f$ is a fibration in Definition \ref{DefSecat}, the maps $s_i$ can be required to be continuous sections of $f$, not only homotopy sections, without changing the value of $\secat(f)$.
	\item Note that our definition of $\secat(f)$ is the reduced version which differs from Schwarz's original definition by one, so that in the case that $f$ admits a global continuous homotopy section, $Y \to X$, by our definition $\secat(f)=0$, while by Schwarz's original definition, it would hold that $\secat(f)=1$. Likewise, we will use the reduced definition of Lusternik-Schnirelmann category as also considered in \cite{CLOT}. 
\end{enumerate}	
\end{remark}
In the following theorem, we collect various basic properties of sectional categories that we will use in the course of this article.

\begin{theorem}
\label{TheoremSecatProperties}
Let $X$ and $Y$ be topological spaces and let $f: X \to Y$. 
	\begin{enumerate}[a)]
		\item $\secat(f:X \to Y)\leq \cat(Y)$, the Lusternik-Schnirelmann category of $Y$.
		\item If $g:X \to Y$ is homotopic to $f$, then $\secat(f)=\secat(g)$.
		\item Let $k \in \NN$, If there are reduced cohomology classes $u_i \in \ker \left[f^*:\widetilde{H}^*(Y;A_i) \to \widetilde{H}^*(X;f^*A_i) \right]$, where $A_i$ is a local coefficient system over $Y$ for each $i \in \{1,2,\dots,k\}$, such that 
 $$u_1 \cup u_2 \cup \dots \cup u_k \neq 0 \in H^*(Y,A_1\otimes A_2 \otimes \dots \otimes A_k), $$
 then $$\secat(f) \geq k.$$
	\end{enumerate}
\end{theorem}

Let $G$ be a group. In the following we will say that a space $X$ is of type $K(G,1)$ if its homotopy groups satisfy
$$\pi_i(X) \cong \begin{cases}
	G & \text{if } i = 1, \\
	\{0\} & \text{if } i \neq 1.
\end{cases}$$
By abuse of notation, we will also write $K(G,1)$ as a space instead of $X$ if our considerations are independent of the chosen space of type $K(G,1)$. 

\begin{definition}
	A group $G$ is called \emph{geometrically finite} if there exists a finite CW complex of type $K(G,1)$.
\end{definition}

By classical algebraic topology, for any two groups $G_1$ and $G_2$ there is a one-to-one correspondence between based homotopy classes of continuous maps $K(G_1,1)\to K(G_2,1)$ and group homomorphisms $G_1 \to G_2$, induced by associating with any continuous $f: K(G_1,1) \to K_2(G_2,1)$ the induced homomorphism $\pi_1(f)$ between the fundamental groups. 

\begin{definition}
Let $G_1$ and $G_2$ be groups and let $\varphi:G_1 \to G_2$ be a group homomorphism. We define the sectional category of $\varphi$ by 
$$\secat(\varphi:G_1 \to G_2) := \secat(f_\varphi:K(G_1,1) \to K(G_2,1)),$$
or just $\secat(\varphi)$, where $f_\varphi$ is a continuous map with $\pi_1(f_\varphi)=\varphi$. By Theorem \ref{TheoremSecatProperties}.b), $\secat(\varphi)$ is well-defined.

Given a group $G$ and a subgroup $H \leq G$, we further let 
$$\secat(H \hookrightarrow G)$$
denote the sectional category of the inclusion of $H$ into $G$ without naming the map explicitly.
\end{definition}
The sectional category of subgroup inclusions was introduced and studied by Z. B\l{}aszczyk, J. Carrasquel Vera and the first author in \cite{BCE22}. Since the cohomology of aspherical spaces can be computed purely algebraically in terms of group cohomology, the properties of sectional category from Theorem \ref{TheoremSecatProperties} can for sectional categories of subgroup inclusions be translated into algebraic statements.

	Both assertions of the following theorem are obtained straightforwardly as special cases of Theorem \ref{TheoremSecatProperties}.a) and c). One uses that by the Eilenberg-Ganea theorem, see \cite{EG65}, and later work of Stallings \cite{Stallings} and Swan \cite{Swan}, it holds for every geometrically finite group $G$ that 
	$$\cat(K(G,1)) = \cd(G),$$
	the cohomological dimension of $G$. See also \cite{Oprea15} for a lucid exposition of the Eilenberg-Ganea results.

 \begin{theorem}
\label{TheoremSecatInclProperties}
	Let $G$ be geometrically finite and let $\iota: H \hookrightarrow G$ be the inclusion of a subgroup. 
	\begin{enumerate}[a)]
		\item $\secat(\iota: H \hookrightarrow G) \leq \cd(G)$. 
		\item If there are reduced cohomology classes $u_i \in \ker[\iota^*:\widetilde{H}^*(G;A_i) \to \widetilde{H}^*(H,\res^G_H(A_i))]$, where $A_i$ is a left $\ZG$-module for each $i \in \{1,2,\dots,k\}$, which satisfy $u_1 \cup u_2 \dots \cup u_k \neq 0$, then 
 $$\secat(\iota:H \hookrightarrow G)\geq k.$$
	\end{enumerate}
\end{theorem}

\begin{remark}
\label{RemarkSecatInclCovering}
Note that for any group $G$ and any subgroup $H \leq G$, $\secat(H \hookrightarrow G)$ coincides with the sectional category of the covering map $K(H,1) \to K(G,1)$ associated with $H$. In particular, this provides an explicit fibration whose sectional category coincides with $\secat(H \hookrightarrow G)$.	
\end{remark}

We end this section by fixing some notational conventions that we will frequently use throughout the article. 

\begin{notation} Let $G$ be a group. 
\begin{itemize}
	\item We let $1 \in G$ denote the unit element and $\ZG$ the integer group ring of $G$. 
	\item We further let $\varepsilon:\ZG\to \ZZ$, $\varepsilon(\sum_{g \in G} n_g \cdot g) = \sum_{g \in G} n_g$, be the augmentation and $K= \ker \varepsilon$ be the augmentation ideal, seen as a left $\ZG$-module. 
	\item Given a subgroup $H \leq G$, we let $\ZZ[G/H]$ be the associated permutation module as a left $\ZZ[G]$-module. We further let $\sigma: \zgh \to \ZZ$, $\sigma (\sum_{x \in G/H} n_x \cdot x) = \sum_{x\in G/H}n_x$ denote its augmentation, and $I = \ker \sigma $ the corresponding augmentation ideal.

\item For any left $\ZG$-module $M$, we put $\widetilde{M}:= \res^G_H(M)$ for the left $\ZH$-module that is obtained via restriction of scalars from $\ZG$ to $\ZH$.
\item We always let $\otimes$ without any subscript denote the tensor product $\otimes_{\ZZ}$ of abelian groups. Given a  $\ZG$-module $M$ and $p \in \NN$ we denote the $p$-fold tensor power of $M$ by 
$$M^p := M^{\otimes p} := M \otimes M \otimes \dots \otimes M$$
and consider it as a $\ZG$-module with respect to the diagonal $G$-action on $M^p$.
\item Given two left $\ZG$-modules $M_1$ and $M_2$ we consider $\Hom_{\ZZ}(M_1,M_2)$, the set of group homomorphisms from $M_1$ to $M_2$, as a left $\ZG$-module via the diagonal $G$-action given by 
$$ (g\cdot f)(m) := gf(g^{-1}m) \qquad \forall m \in M_1, \ g \in G, \ f \in \Hom_{\ZZ}(M_1,M_2).$$
\end{itemize}
\end{notation}
Note that as free abelian group, 
$$K= \bigoplus_{g \in G} \ZZ\cdot (g-1), \qquad I = \bigoplus_{gH \in G/H}\ZZ \cdot (gH-H), $$
where $H \in G/H$ denotes the class of the unit element.

% ... 

\section{Essential classes relative to subgroups}
\subsection{The definition of essential classes} Let $G$ be a group and let $i:K \hookrightarrow \ZG$ be the inclusion of the augmentation ideal. It is shown in \cite{DranishRudyak09} and discussed in \cite{BCE22} that we obtain a projective resolution of $\ZZ$ over $\ZG$ via
\begin{equation}
\label{EqProjRes}
\dots \to \ZG \otimes K^s \xrightarrow{p_s} \ZG \otimes K^{s-1} \to \dots \to \ZG \otimes K \xrightarrow{p_1} \ZG \xrightarrow{\varepsilon} \ZZ \to 0, 
\end{equation}
where 
$$p_s: \ZG \otimes K^s \to \ZG \otimes K^{s-1}, \quad p_s(x \otimes y \otimes z) = \varepsilon(x) \cdot i(y) \otimes z \quad \forall x \in \ZG, \ y \in K, \ z \in K^{s-1}.$$
In the following, we shall use this resolution in our study of cohomology groups of $G$. 

Let us recall the following definition from \cite{BCE22}, which will play a crucial role throughout the rest ot this article.

\begin{definition}
\label{DefBSrelative}
Let $H \leq G$ be a subgroup. We define the \emph{Berstein-Schwarz class of $G$ relative to $H$} as the class $\omega \in H^1(G;I)$ represented by the cocycle $$\xi\in \Hom_{\ZG}(\ZG\otimes K,I), \qquad \xi = \mu \circ (\varepsilon \otimes \id_K),$$
where $\mu\colon K \rightarrow I$ is induced by the canonical projection $G \rightarrow G/H$. 
\end{definition}

We also recall the description of the cup product in group cohomology with respect to the resolution \eqref{EqProjRes} provided in \cite[Proposition 1.6]{BCE22}. For cohomology classes $[a]\in H^p(G;A)$ and $[b]\in H^q(G;B)$ represented by cocycles $a\colon \ZG\otimes K^{ p}\to A$ and $b\colon\ZG\otimes K^{ q}\to B$, the cup product $[a]\cup [b]\in H^{p+q}(G;A\otimes B)$ is represented by the map \[\ZG\otimes K^{ p+q}\xrightarrow{\varepsilon\otimes\id}K^{p+q}\xrightarrow{\hat{a}\otimes\hat{b}} A\otimes B\] where $\hat{a}:K^{ p}\to A$ the $\ZG$-module homomorphism that induced in the diagram 
\[\begin{tikzcd}
	K^{p+1}\ar[r]&\ZG \otimes K^{p}\ar[d, "a"]\ar[r]&K^{ p}\ar[dl, "\hat{a}"] \\
	&A,
\end{tikzcd}\]
since $a$ vanishes on $K^{p+1}$, and analogously for $\hat{b}$. Under this characterization, the $n$-th power of the relative Berstein-Schwarz class $\omega^n\in H^n(G;I^{ n})$ is represented by the map \[\ZG\otimes K^{ n}\xrightarrow{\varepsilon\otimes\id} K^{ n}\xrightarrow{\mu^{ n}} I^{n}.\]
The next proposition relates the powers of relative Berstein-Schwarz classes to sectional category.

\begin{proposition}
\label{PropSecatOmega}
Let $H \leq G$ be a subgroup and let $\omega \in H^1(G;I)$ be the Berstein-Schwarz class of $G$ relative to $H$. Then
	$$\secat( H \hookrightarrow G) \geq \height(\omega) = \sup \{n \in \NN\ | \ \omega^n \neq 0 \}.$$
\end{proposition}
\begin{proof}
Let $i: H \hookrightarrow G$ be the inclusion and let $k := \height(\omega)$. By \cite[Lemma 2.6]{BCE22}, it holds that $\omega \in \ker \iota^*$. Thus, it follows immediately from Theorem \ref{TheoremSecatInclProperties}.b) by taking $u_i = \omega$ for each $i \in \{1,2,\dots,k\}$.
\end{proof}
The role of relative Berstein-Schwarz classes for sectional categories of subgroup inclusions is analogous to the role of Berstein-Schwarz classes for Lusternik-Schnirelmann category and the role of the so-called canonical classes for topological complexity. The latter were introduced by A. Costa and M. Farber in \cite{Costa}. We shall discuss  below how canonical classes are related to relative Berstein-Schwarz classes.


The following lemma provides an alternative characterization of relative Berstein-Schwarz classes and is an analogue of \cite[Lemma 5]{Costa}.

\begin{lemma}
	Consider the short exact sequence of $G$-modules 
\begin{equation}
\label{EqShortExact}
	0 \to I \stackrel{i}\to \zgh \stackrel{\sigma}\to \ZZ \to 0
	\end{equation}
	and let $\delta: H^0(G;\ZZ) \to H^1(G;I)$ denote the Bockstein homomorphism	associated with that sequence. Then 
	$$\omega = \delta(1),$$
	where $1 \in H^0(G;\ZZ)$ is a generator.
\end{lemma}
\begin{proof}
	Let $\rho: \ZG \to \zgh$ denote the homomorphism induced by the orbit space projection. We know that $\omega \in H^1(G;I)$ is represented by $f:\ZG \otimes K\to I$, $f= \rho \circ (\varepsilon \otimes \id_K)$. We consider the long exact sequence associated with 
	$$0 \to C^0(G;I) \to C^0(G;\zgh) \to C^0(G;\ZZ)\to 0.$$
	In terms of our resolution, its connecting homomorphism, i.e. the Bockstein homomorphism, is obtained via diagram chasing in 
	$$ \small{\begin{tikzcd}[cramped]
		0 \arrow[r]& \Hom_{\ZG}(\ZG,I) \arrow[d, "d_1^*"] \arrow[r, "i_*"]& \Hom_{\ZG}(\ZG,\zgh)\arrow[r, "\sigma_*"] \arrow[d, "d_1^*"] & \Hom_{\ZG}(\ZG,\ZZ)\ar[r] \arrow[d, "d_1^*"]& 0 \\
		0 \arrow[r]& \Hom_{\ZG}(\ZG\otimes K,I) \ar[r, "i_*"]& \Hom_{\ZG}(\ZG\otimes K,\zgh)\ar[r, "\sigma_*"] & \Hom_{\ZG}(\ZG\otimes K,\ZZ)\ar[r] & 0
	\end{tikzcd}} $$  
	The augmentation $\varepsilon \in \Hom_{\ZG}(\ZG,\ZZ)$ is a cocycle. By definition of the maps involved, it holds that $\varepsilon=\sigma \circ \rho$, i.e. $\sigma_*(\rho)=\varepsilon$. Diagram chasing shows that
	$$i_*(f)= d_1^*(\rho), \qquad $$
	so by definition of the Bockstein homomorphism, we obtain that 
	$$\delta(1) = \delta([\varepsilon]) = [f]=\omega.$$
	Here, one sees that $1=[\varepsilon]$ generates $H^0(G;\ZZ)$ as $\varepsilon(g)=1$ for each  $g \in G$.   
\end{proof}

By elementary homological algebra, tensoring the short exact sequence \eqref{EqShortExact} with a left $\ZG$-module $M$ that is $\ZZ$-free yields a short exact sequence of $\ZG$-modules 
\begin{equation}
\label{EqSESmodule}
0 \to I \otimes M \to \zgh \otimes M \to M \to 0
	\end{equation}
	with respect to the diagonal $G$-actions. In complete analogy with a statement for the canonical class observed in Section 3 of \cite{FarMes20}, we derive the following statement.
	
\begin{corollary}
	\label{CorBockstein}
	Let $M$ be a $\ZZ$-free left $\ZG$-module and let $u \in H^i(G;M)$, where $i \in \NN_0$. Consider the Bockstein homomorphism $\delta$ of the coefficient sequence
	$$0 \to I \otimes M \to \zgh \otimes M \to M \to 0.$$
	Then 
	$$\delta(u) = \omega \cup u \in H^{i+1}(G;I \otimes M).$$
\end{corollary}
\begin{proof}
It follows straight from \cite[V.(3.3)]{Brown82}	 and the graded commutativity of the cup product that
	$$\delta(u) = \delta(u \cup 1) = (-1)^i u \cup \delta(1)=(-1)^i u \cup \omega = \omega \cup u.$$
\end{proof}
We  let $\Hom_{\ZZ}(I^s,M)$ be equipped with the diagonal $G$-action and consider
$$\ev_s: I\otimes \Hom_{\ZZ}(I^{s+1},M) \to \Hom_{\ZZ}(I^s,M),\qquad  \ev( x \otimes f) = f(x \otimes \cdot),$$
which is seen to be a $\ZG$-homomorphism. The following statement and its proof are straightforward analogues and carried out along the lines of \cite[Proposition 7.3]{FarMes20}. 

\begin{proposition} \label{Bocksteinev}
	Let $A$ be a left $\ZG$-module. For any cohomology class $u \in H^r(G;\Hom_{\ZZ}(I^{s+1},A))$ one has
	$$\delta(u)	=-(\ev_s)_*(\omega \cup u), $$
	where $\delta$ is the Bockstein homomorphism associated with the short exact coefficient sequence
	\begin{equation}		
\label{EqDualSES}
	0\to \Hom_{\ZZ}(I^s,A) \stackrel{\sigma^*}\to \Hom_{\ZZ}(\zgh \otimes I^s,A) \stackrel{i^*}{\to} \Hom_{\ZZ}(I^{s+1},A) \to 0
	\end{equation}
	obtained by applying $\Hom_{\ZZ}(\cdot,A)$ to \eqref{EqSESmodule} in the case of $M=I^s$.
\end{proposition}
\begin{proof}
We first observe that \eqref{EqDualSES} is indeed short exact, since $I^s$, $\zgh \otimes I^s$ and $I^{s+1}$ are all $\ZZ$-free. Let $\beta:H^r(G;\Hom_{\ZZ}(I^{s+1},A))\to H^{r+1}(G;I \otimes \Hom_{\ZZ}(I^{s+1},A))$ denote the Bockstein homomorphism of the short exact coefficient sequence 
	$$0\to I \otimes \Hom_{\ZZ}(I^{s+1},A) \hookrightarrow \zgh \otimes \Hom_{\ZZ}(I^{s+1},A) \stackrel{\sigma \otimes \id }\to\Hom_{\ZZ}(I^{s+1},S) \to 0$$
	obtained by letting $M=\Hom_{\ZZ}(I^{s+1},A)$ in \eqref{EqSESmodule}.
	By Corollary \ref{CorBockstein}, it holds that
	$$\beta(u) = \omega \cup u.$$ 
	Thus, the claim immediately follows if we can show that $\delta = -(\ev_s)_* \circ \beta$. We first consider the homomorphism  
	$F: \zgh \otimes \Hom_{\ZZ}(I^{s+1},A) \to \Hom_{\ZZ}(\zgh \otimes I^s,A)$ given by $\ZZ$-linearly extending 
	$$(F(xH \otimes f))(zH \otimes y) = f((zH-xH)\otimes y) \qquad \forall x,z \in G, \ y \in I^s, \ f \in \Hom_{\ZZ}(I^{s+1},A).$$
	We compute that
	\begin{align*}
		(F(g \cdot (xH\cdot f)))(zH \otimes y) &= F(gxH\otimes (g \cdot f))(zH \otimes y)\\
		&= (g \cdot f)((zH-gxH)\otimes y) = g f((g^{-1}zH - xH)\otimes g^{-1}y) \\
		&= g (F(x\otimes f))(g^{-1}zH\otimes g^{-1}y) =  ( g \cdot F(x\otimes f))(zH \otimes y)
	\end{align*}
	for all $g,x,z\in G$, $y \in I^s$ and $f \in \Hom_{\ZZ}(I^{s+1},M)$. Hence, $F$ is a $\ZG$-homomorphism. Consider the following diagram with exact rows:
	$$ \begin{tikzcd}[cramped]
		0 \ar[r] & I \otimes \Hom_{\ZZ}(I^{s+1},M) \ar[d,"\ev_s"]\ar[r, "i \otimes \id"] & \zgh \otimes \Hom_{\ZZ}(I^{s+1},M) \ar[d, "F"] \ar[r, "\sigma \otimes \id"]& \Hom_{\ZZ}(I^{s+1},M)\ar[d, "\id"]\ar[r] & 0 \\
		0 \ar[r] &  \Hom_{\ZZ}(I^{s},M) \ar[r, "-\sigma^*"] &  \Hom_{\ZZ}(\zgh\otimes  I^{s},M) \ar[r,"i^*"]& \Hom_{\ZZ}(I^{s+1},M) \ar[r]& 0 
	\end{tikzcd} $$
	To show that the left-hand square of this diagram commutes, we compute for all $x,z\in G$, $y \in I^s$ and $f \in \Hom_{\ZZ}(I^{s+1},M)$ that
	$$((-\sigma^* \circ \ev_s)((xH-H)\otimes f))(zH\otimes y)=-\sigma(zH)\cdot f((xH-H)\otimes y)=- f((xH-H)\otimes y).$$
	and
	\begin{align*}
		&((F \circ (i \otimes \id))((xH-H)\otimes f))(zH\otimes y) = (F(xH\otimes f))(zH \otimes y) - (F(H\otimes f))(zH\otimes y)\\
		&=f((zH-xH)\otimes y) - f((zH-H)\otimes y) = -f((xH-H)\otimes y).
	\end{align*}
	Comparing the results shows the commutativity of the left-hand square. Concerning the right-hand square, we derive that
	\begin{align*}
		&((i^*\circ F)(xH \otimes f ))((zH-H)\otimes y) =  (F(xH \otimes f))(zH\otimes y) - (F(xH \otimes f))(H\otimes y) \\
		&= f((zH-xH)\otimes y)- f((H-xH)\otimes y) \\
		&= f((zH-H)\otimes y) = \sigma(xH) \cdot f((zH-H)\otimes y) = ((\sigma \otimes \id)(x \otimes f))((zH-H)\otimes y).
	\end{align*}
	Thus, the above diagram commutes. Considering the long exact cohomology sequences associated with the coefficient groups of the above diagram, the naturality of Bockstein homomorphisms shows that 
	$$-(\ev_s)_* \circ \beta  = \delta \circ \id_* = \delta,$$
	which we wanted to show. The claim immediately follows. Here, the additional sign stems from the fact that we have considered $-\sigma^*$ instead of $\sigma^*$ in the bottom row of the diagram. 
\end{proof}
We introduce some additional terminology which generalizes the notion of essential classes introduced in \cite{FarMes20}.

\begin{definition}
Let $n \in \NN$ and let $\alpha \in H^n(G;A)$ with $\alpha \neq 0$. We say that $\alpha$ is \textit{essential relative to $H$} if there exists a homomorphism of $\Z[G]$-modules $\varphi \colon I^{n} \rightarrow A$, such that \[ \varphi_*(\omega^n)=\alpha. \] 
\end{definition} 
\begin{remark}
\label{RemarkEssential}
Cohomology classes which are essential relative to subgroups are used to derive lower bounds on the sectional category of the corresponding subgroup inclusion: Assume that for some $n \in \NN$, there exists a class $u \in H^n(G;A)$ with $u \neq 0$ that is essential relative to $H$. By definition of essential classes, this requires that $\omega^n \neq 0 \in H^n(G;I^n)$, which in turn yields that $\secat(H \hookrightarrow G) \geq n$ by Proposition \ref{PropSecatOmega}. \medskip 
\end{remark}

\subsection{Essential classes relative to normal subgroups}

To close this section, let us consider the case of the inclusion of a normal subgroup. In this setting we can characterize essential classes relative to that subgroup as pullbacks of non-trivial classes in the cohomology of the quotient group through the homomorphism in cohomology induced by the quotient map. This is, in certain measure, a generalization of the ideas present in the case of the inclusion of the diagonal subgroup in abelian groups, as considered in \cite[Section 6]{FarMes20}.

\begin{proposition}
\label{PropSecatNormal}
Let $N \triangleleft G$ be a normal subgroup, put $Q := G/N$ for the quotient group and let $\pi: G \to Q$ denote the projection. 
\begin{enumerate}[a)]
\item Let $\omega \in H^1(G;I)$ be the Berstein-Schwarz class of $G$ relative to $N$ and let $\beta \in H^1(Q;I_Q)$ be the Berstein-Schwarz class of $Q$, where $I_Q \subset \ZZ[Q]$ denotes the augmentation ideal of $Q$. Then
$$\pi^*\beta = \omega.$$
	\item Let $A$ be a left $\ZZ[Q]$-module and let $n \in \NN$. A cohomology class $u \in H^n(G;\pi^*A)$ with $u \neq 0$ is essential relative to $N$ if and only if there exists $v \in H^n(Q;A)$ with $\pi^*v=u$.
\end{enumerate}
\end{proposition}
\begin{proof}
Throughout the proof, we will use the projective resolution $(\ZG\otimes K^*,p_*)$ of $\ZZ$ over $\ZG$ and the projective resolution $(\ZZ[Q]\otimes I_Q^*,p_*)$ of $\ZZ$ over $\ZZ[Q]$, both defined as in the beginning of this section, to compute the cohomology groups of $G$ and $Q$, respectively. By abuse of notation, we further denote the ring homomorphism induced by $\pi$ by $\pi:\ZG \to \ZZ[Q]$ as well.
	\begin{enumerate}[a)]
		\item By definition of the augmentation ideals and of $\pi$, it holds that $\pi(K) \subset I_Q$ and in the notation of Definition \ref{DefBSrelative}, we write  $\mu:= \pi|_K: K \to I_Q$. By factorwise applying $\pi$, one obtains a chain map 
		$$\pi_\#: \ZG \otimes K^* \to \ZZ[Q]\otimes I_Q^*,$$
This map induces a cochain map
	$$(\pi_\#)^*: \Hom_{\ZZ[Q]}(\ZZ[Q]\otimes I_Q^*,A) \to \Hom_{\ZG}(\ZG \otimes K^*,\pi^*A),$$
	which in turn induces the pullback map $$ \pi^*:H^*(Q;A)\to H^*(G;\pi^*A). $$

Denote the augmentations of $G$ and $Q$ by $\varepsilon_G$ and $\varepsilon_Q$, respectively. As shown in \cite{DranishRudyak09}, the Berstein-Schwarz class $\beta \in H^1(Q;I_Q)$ is then represented by the cocycle 
		$$f_\beta: \ZZ[Q]\otimes I_Q \to I_Q, \qquad f_\beta = \varepsilon_Q \otimes \id_{I_Q}.$$
	One easily checks that $\varepsilon_Q \circ \pi = \varepsilon_G$, so that 
	$$(\pi_\#)^*(f_\beta) (x)(y) = (\varepsilon_Q \circ \pi)(x) \cdot \pi_\#(y) = \varepsilon_G(x)\cdot \mu(y)= (\mu \circ (\varepsilon_G\otimes \id_K))(x \otimes y).$$
	for all $x \in \ZG$ and $y \in K$. By definition of $\omega$, it is represented by this latter cocycle. Passing to cohomology then shows that $\pi^*\beta =\omega$.

		\item  Assume that $u$ is essential relative to $N$, such that there exists a $\ZG$-homomorphism $\varphi:I^n \to \pi^*A$ with 
 $$u = \varphi_*(\omega^n)= \varphi_*((\pi^*\beta)^n) = (\varphi_* \circ \pi^*)(\beta^n).$$
  One easily checks that $\pi^*I_Q = I$ as $\ZG$-modules. Moreover, since we can view $\varphi:I_Q^n \to A$ as a $\ZZ[Q]$-homomorphism and since the diagram 
\begin{equation}
\label{EqCD}
\begin{CD}
 	H^n(Q;I_Q^n) @>{\varphi_*}>> H^n(Q;A)\\
 	@V{\pi^*}VV @V{\pi^*}VV\\ 
H^n(G;I^n) @>{\varphi_*}>>H^n(G;\pi^*A)
 \end{CD} 
 \end{equation}
 obviously commutes, we obtain that $u = \pi^*v$, where $v:= \varphi_*(\beta^n) \in H^n(Q;A)$.
 
 	Conversely, assume that there exists a class $v \in H^n(Q;A)$, for which $u=\pi^*v$. By the universality of Berstein-Schwarz classes, see \cite{DranishRudyak09}, there exists a $\ZZ[Q]$-homomorphism 
	$$\psi:I_Q^n \to A, $$
	such that $v = \psi_*(\beta^n)$. In fact, by definition of pullback modules, we can view $\psi$ as a $\ZG$-homomorphism $$\psi:I^n = \pi^*I_Q^n \to \pi^*A.$$
	Replacing $\varphi_*$ by $\psi_*$, the diagram corresponding to \eqref{EqCD} commutes as well, so we obtain that 
	$$u = \pi^*(\psi_*(\beta^n))= \psi_*(\pi^*(\beta^n))= \psi_*(\omega^n),$$
	hence $u$ is essential.
		\end{enumerate}
\end{proof}

We want to derive an estimate for $\secat(N \hookrightarrow G)$ from the previous proposition, for which we need to introduce another notion. Recently, Mark Grant defined the cohomological dimension $\cd(\phi)$ of a group homomorphism $\phi\colon G \to H$ to be the
maximum $k$ for which there exists some $H$-module $A$ so that $\phi^*\colon H^k(H;A) \to H^k(G;\phi^*A)$ is non-trivial. The first published account of the study of this new dimension is even more recent, see \cite{DK}. \medskip  

For the proof of the following theorem we recall that the LS-category of a map $f\colon X \to Y$ is the smallest integer $m$ for which there are $m+1$ open sets $U_0,\ldots, U_m$ which cover $X$ and such that each of the restrictions $f|_{U_j}$ is nullhomotopic. For a group
homomorphism $\phi\colon G \to H$, we write $\cat(\phi)$ for the category of the associated map of aspherical spaces
$K(G,1) \to K(H,1)$.

\begin{theorem}
\label{TheoremSecatNormal}
	Let $N \triangleleft G$ be a normal subgroup, put $Q := G/N$ for the quotient group and let $\pi: G \to Q$ denote the projection. Then 
	$$\cd(\pi:G \to Q) \leq \secat(N \hookrightarrow G) \leq \cd(Q).$$
	In particular, if $\pi^*:H^{\cd(Q)}(Q;A) \to H^{\cd(Q)}(G;\pi^*A)$ is non-zero for some $\Z[Q]$-module $A$, then $\secat(N \hookrightarrow G) = \cd(Q)$.
\end{theorem}
\begin{proof}
Put $k:= \cd(\pi:G \to Q)$ and let $A$ be a left $\Z[Q]$-module and $u \in H^k(Q;A)$ with $\pi^*u \neq 0$. Then, by Proposition \ref{PropSecatNormal}, $\pi^*u\in H^k(G;\pi^*A)$ is essential relative to $N$. This in particular yields  that $\omega^k \neq 0$, where $\omega \in H^1(G;I)$ denotes the Berstein-Schwarz class of $G$ relative to $N$ and it follows from Proposition \ref{PropSecatOmega} that $\secat(N \hookrightarrow G) \geq k$. 
	 
To show the other inequality we can argue through properties of Lusternik-Schnirelmann category as
found in \cite{CLOT}.  First, let's note that the exact sequence $$\{1\} \to N \stackrel{i}{\to} G \stackrel{\pi}{\to} Q \to \{1\}$$ gives a fibre sequence
$$K(N,1) \stackrel{i}{\to} K(G,1) \stackrel{\pi}{\to} K(Q,1)$$ where we have used the same notation for the space maps. A fibre
sequence arises as a homotopy pullback
$$\begin{tikzcd}
	K(N,1) \ar[r] \ar[d, swap, "i"] & P_0(K(Q,1)) \ar[d] \\
	K(G,1) \ar[r, "\pi"] & K(Q,1)
\end{tikzcd} $$
where $P_0(K(Q,1)) \to K(Q,1)$ is the based path space fibration. By \cite[Proposition 9.18]{CLOT}, because $P_0(K(Q,1))$ is
contractible, we have $\secat(i\colon N \hookrightarrow G) = \cat(\pi)$. But a standard property of the category of a map is that it is bounded above by both the category of its domain and the category of its codomain. 
Hence, \begin{equation} \label{catmapcdQ}
	\secat(i \colon N \hookrightarrow G) = \cat(\pi) \leq \cat(K(Q,1))=\cd(Q).
\end{equation} 
Assume now the hypothesis that $\pi^*:H^{\cd(Q)}(Q;A) \to H^{\cd(Q)}(G;\pi^*A)$ is non-zero for some $\Z[Q]$-module $A$. Then, by definition, $\cd(\pi) \geq \cd(Q)$. Combining this with the lower bound by $\cd(\pi)$ and with inequality \ref{catmapcdQ} shows that $$ \secat(N \hookrightarrow G) = \cd(\pi) = \cd(Q).$$
\end{proof}

\begin{remark}\label{rem:alt}
\begin{enumerate}[(1)]
\item Notice that if we assume $\cd(Q) \neq 2$ (and thus we remove the pathological case prescribed by the Eilenberg-Ganea conjecture) one could also argue in the proof of the upper bound of Theorem \ref{TheoremSecatNormal} as follows: observe first that, by \cite[Corollary 2.4]{BCE22}, it holds that $$ \secat(N \hookrightarrow G) \leq \dim E_{\left<N\right>}G$$ where $E_{\left<N\right>}G$ is the classifying space of the family of groups generated by $N$. Since $N$ is a normal subgroup of $G$, we derive from \cite[Corollary 4.22]{ArcinCisner17} that $\dim (E_{\left<N\right>}G)=\dim (K(Q,1))$, where $K(Q,1)$ is a classifying space of $Q$. Using the Eilenberg-Ganea theorem and Theorem \ref{TheoremSecatInclProperties}.a) we derive that
$$\secat(N \hookrightarrow G) \leq \dim(K(Q,1))= \cd(Q).$$ Combining this with the first inequality of Theorem \ref{TheoremSecatNormal} shows the claim.
\item Since $\cat(K(Q,1))=\cd(Q)$ for any $Q$ by the Eilenberg-Ganea theorem, there arose the natural conjecture that
$\cat(\phi)=\cd(\phi)$ for any homomorphism $\phi\colon G \to H$. This was disproved by T. Goodwillie using an infinitely generated 
group $G$. In \cite[Theorem 5.4]{DK} a finitely generated example was derived and we shall use this in Example \ref{exam:nonness}. 
\end{enumerate}
\end{remark}

\begin{example}\label{exam:nonness}
While the hypothesis that $\cd(\pi:G \to Q)=\cd(Q)$ on cohomology in Theorem \ref{TheoremSecatNormal} is sufficient to derive $\secat(N \hookrightarrow G) = \cd(Q)$, it is not necessary. We can see this using \cite[Theorem 5.4]{DK} as follows. We recall that in \cite{Bolotov}, D. Bolotov defined a closed manifold $M^4$ with fundamental group $\pi_1(M) = \Z * \Z^3$ for which, as shown in \cite{DK}, the pullback map $$ \mu^*\colon H^3(K(\Z * \Z^3,1);A) \to H^3(M;A)$$ is the zero homomorphism for
all $\Z * \Z^3$-modules $A$, where $\mu\colon M \to K(\pi,1)$ is a classifying map of the universal cover. The hyperbolization procedure of \cite{CD}
gives a closed aspherical manifold $W^4$ and a degree one map $\alpha\colon W \to M$ which induces a surjection of fundamental groups.
The surjective group homomorphism $G=\pi_1(W) \to \Z*\Z^3$ is then induced by the composition $$ W \stackrel{\alpha}{\to} M
\stackrel{\mu}{\to} K(\Z * \Z^3,1)=S^1 \vee T^3$$ and we have a map $\theta\colon W \to T^3$ given by the composition
$$W \stackrel{\alpha}{\to} M \stackrel{\mu}{\to} S^1 \vee T^3 \stackrel{c}{\to} T^3$$
where $c\colon S^1 \vee T^3 \to T^3$ collapses $S^1$. Abusing notation, the induced homomorphism of fundamental groups 
$\theta\colon G \to \Z^3$ is also a surjection. Letting $N=\ker\theta$, we have an exact sequence 
$\{1\} \to N \stackrel{i}{\to} G \stackrel{\theta}{\to} \Z^3 \to \{1\}$ .
In \cite[Theorem 5.4]{DK} it is shown that
$$\cat(\theta) =\cd(\Z^3)=3.$$
As in Remark \ref{rem:alt}, this means that $\secat(N \hookrightarrow G)=3$ as well. However, the fact that $\mu^*=0$ and that $\theta = c\circ \mu\circ \alpha$ shows that the map $\theta^*\colon H^3(\Z^3;A) \to H^3(G;A)$ is trivial for any coefficient module, hence $\cd(\theta) < 3$. 
\end{example}

\section{Forming the spectral sequence and deriving a lower bound}

In this section we will proceed to generalize the construction of a spectral sequence to sectional categories of subgroup inclusions that has been carried out for the topological complexity of aspherical spaces by the second and third authors in \cite[Section 7]{FarMes20}. In our setting, the spectral sequence fom \cite{FarMes20} corresponds to the choice of $G=\pi \times \pi$ and $H= \Delta_\pi$, for a given group $\pi$.  The steps of the construction are carried out in complete analogy with the corresponding parts of \cite{FarMes20} and  instead of giving individual references for each statement, we view this as a general reference to \cite[Section 7]{FarMes20}. The interested reader will have no difficulties in finding the analogous statements therein. 
 
\subsection{The construction of the spectral sequence} Let $G$ be a group, let $H \leq G$ be a subgroup and let $\omega \in H^1(G;I)$ be the Berstein-Schwarz class of $G$ relative to $H$.  Let $A$ be a left $\ZG$-module. Define the groups $$ E_0^{r,s} = \Ext^r_{\ZG}(\Z[G/H] \otimes I^s, A), \qquad D_0^{r,s} = \Ext^r_{\ZG}(I^s,A) \qquad \forall r,s\in \NN_0. $$

Let $i: I \hookrightarrow \zgh$ denote the inclusion. For each $s\in \NN$ the short exact sequence from \eqref{EqSESmodule} with $M=I^s$ yields a short exact sequence of $\ZG$-modules
\begin{equation}
	\label{ShortExactSeq}
	 0 \rightarrow I^{s+1} \xrightarrow{f_s} \Z[G/H] \otimes I^s \xrightarrow{g_s} I^s \rightarrow 0 ,
\end{equation}
where $f_s:I^{s+1} \to \zgh \otimes I^s$, $f_s := i \otimes \id_{I^s}$ and $g_s:\zgh\otimes I^s \to I^s$, $g_s(x \otimes y) = \sigma(x) \cdot y$.

For each $s$, the sequence in \eqref{ShortExactSeq} induces a long exact $\Ext$-sequence with coefficients in $A$, which is in the above notation given as
\begin{equation}
	\label{CompSpSe}
	\cdots \rightarrow E_0^{r,s} \xrightarrow{k_0} D_0^{r,s+1} \xrightarrow{i_0} D_0^{r+1,s} \xrightarrow{j_0} E_0^{r+1,s} \rightarrow \cdots
\end{equation} 
where 
\begin{itemize}
	\item $i_0:D^{r,s+1}_0\to D^{r+1,s}_0$ denotes the connecting homomorphism,
	\item $j_0: D^{r,s}_0 \to E^{r,s}_0$ is induced by $(g_s)^*: \Hom_{\ZG}(I^s,A) \to \Hom_{\ZG}(\zgh\otimes I^s,A)$,
	\item $k_0: E^{r,s}_0 \to D^{r,s+1}_0$ is induced by $(f_s)^*:\Hom_{\ZG}(\zgh \otimes I^s,A) \to \Hom_{\ZG}(I^{s+1},A)$.
	\end{itemize}
We put
$$ E_0 := \bigoplus_{r,s \in \NN_0} E_0^{r,s} = \bigoplus_{r,s \in \NN_0} \Ext^r_{\ZG}(\Z[G/H] \otimes I^s,A)$$ and $$ D_0 := \bigoplus_{r,s \in \NN_0} D_0^{r,s} = \bigoplus_{r,s \in \NN_0}\Ext^r_{\ZG}(I^s,A) $$ 
and consider the summandwise defined maps
$$i_0:D_0 \to D_0, \qquad j_0: D_0 \to E_0, \qquad k_0: E_0 \to D_0.$$
Together with these maps the groups $D_0$ and $E_0$ form an exact couple
\[ \begin{tikzcd}
	D_0 \arrow[rr, "i_0"] & & D_0. \arrow[dl, "j_0"] \\
	& E_0 \arrow[lu, "k_0"]
\end{tikzcd} \] 
For each $p \in \NN$ we denote its $p$-th derived exact couple as
	\[ \begin{tikzcd}
	D_p \arrow[rr, "i_p"] & & D_p \arrow[dl, "j_p"] \\
	& E_p \arrow[lu, "k_p"]
\end{tikzcd} \] where, for each $ p \in \NN$ the module $D_p^{r,s}$ is defined as the image of $p$ compositions of the map $i_0$, i.e. 
$$ D_p^{r,s} = \Img(i_{p-1} \colon D_{p-1}^{r-1,s+1} \rightarrow D_{p-1}^{r,s}) = \Img(\underbrace{i_0\circ  \cdots \circ i_0}_{p} \colon D_0^{r-p,s+p} \rightarrow D_0^{r,s})  $$ 
and, naturally, the module $\E_p^{\ast, \ast}$ is defined by taking cohomology with respect to the differential defined by the exact couple, that is $$ E_p^{\ast, \ast} = H^*(E_{p-1}^{\ast, \ast}, d_{p-1}). $$ The zeroth page of the first-quadrant cohomological spectral sequence obtained thereby is formed by the groups $E^{r,s}_0$ and the differential 
$$d_0: E^{r,s}_0 \to E^{r,s+1}_0, \qquad d_0 := j_0 \circ k_0 \qquad \forall r,s \in \NN_0.$$  Note that we can view $D^{r,s}_p \subset D^{r,s}_0$ as subsets for each $p \in \NN$ and we will occasionally do so without further mention.\medskip 

Let $n,p \in \NN$ with $p \leq n$. Taking a class $\alpha \in D_p^{n,0}$ we know by definition that $\alpha = i_0^p(\gamma)$ for some $\gamma \in D_0^{n-p,p}$. By \cite[Proposition III.2.2]{Brown82}, we can identify $$ D_0^{n-p,p} = \Ext_{\ZG}^{n-p}(I^p,A) \cong H^{n-p}(G; \Hom_{\Z}(I^p,A)).$$ 
Following an iterated use of the identification provided by Proposition \ref{Bocksteinev}, we obtain the following characterization of $D_p^{n,0}$, which is a generalization of \cite[Corollary 7.4]{FarMes20}.

\begin{proposition}
\label{PropPushfwdBerstein}
Let $n,p \in \NN$ with $p \leq n$ and let $\alpha \in D^{n,0}_0$. Then $\alpha \in D^{n,0}_p$ if and only if there exists $\gamma \in H^{n-p}(G;\Hom_{\ZZ}(I^p,A))$ with
$$\alpha = \psi_*(\omega^p \cup \gamma), $$ where $\psi \colon I^p \otimes \Hom_{\Z}(I^p,A) \rightarrow A$ is the $\ZG$-homomorphism given by$$ \psi(x_1 \otimes \cdots \otimes x_p \otimes f) = f(x_p \otimes x_{p-1} \otimes \cdots \otimes x_1) .$$
\end{proposition}

This proposition has an immediate consequence for sectional categories. 

\begin{theorem}
\label{TheoremSecatDp}
Let $n,p \in \NN$ with $p \leq n$. If $D^{n,0}_p \neq \{0\}$, then $\omega^p \neq 0$ and thus
$$\secat(H \hookrightarrow G) \geq p.$$
\end{theorem}
\begin{proof}
By Proposition \ref{PropPushfwdBerstein}, every class in $D^{n,0}_p$ is obtained as a pushforward of a cup product of $\omega^p$ with another class. So if there is a non-trivial class in $D^{n,0}_p$, then it necessarily holds that $\omega^p \neq 0$ and the claim follows from Proposition \ref{PropSecatOmega}.
\end{proof}

\subsection{Essential classes and the spectral sequence} To extract further consequences for $\secat(H \hookrightarrow G)$ from the spectral sequence, we need to introduce some auxiliary lemmas on the groups $E^{r,s}_0$.

\begin{lemma}
	\label{LemmaIsomHom}
	Let $M$ and $N$ be left $\ZG$-modules. Let $\ZG \otimes M$ be equipped with the diagonal $G$-action. Then
	$$\Phi: \Hom_{\ZG}(\Z[G/H]\otimes M, N) \to \Hom_{\ZH}(\widetilde{M},\widetilde{N}), \quad (\Phi(f))(x):= f( H\otimes x) \quad \forall x \in M, $$
	is an isomorphism of abelian groups. 
\end{lemma}
\begin{proof}
	It is easy to see that $\Phi$ is a well-defined group homomorphism. Consider the map 
$$\Psi: \Hom_{\ZH}(\widetilde{M},\widetilde{N}) \to \Hom_{\ZG}(\Z[G/H]\otimes M, N),\quad (\Psi(f))( H \otimes x) = g f(g^{-1}x).$$
	For each $f \in \Hom_{\ZH}(\widetilde{M},\widetilde{N})$ and $h \in H$ we obtain that
	$$gh f(h^{-1}g^{-1}x)= g f(g^{-1}x) \quad \forall g \in G, x \in M,$$
	since $f$ is a $\ZH$-homomorphism. This shows that $\Psi(f)(gH \otimes m)$ is independent of the chosen representative of $gH$, thus $\Psi(f): \Z[G/H]\otimes M \to N$ well-defined.
	
	For all $f \in \Hom_{\ZH}(\widetilde{M},\widetilde{N})$, $g_1,g_2 \in G$ and $x \in M$ we further compute that 
	\begin{align*}
		(\Psi(f))(g_1 \cdot g_2H \otimes x) &= \Psi(f)(g_1g_2H \otimes g_1x)=  g_1g_2 f(g_2^{-1}g_1^{-1}g_1x)\\
		&= g_1 \cdot g_2 f(g_2^{-1}x) = g_1 \cdot (\Psi(f))(g_2H\otimes x),
	\end{align*}
	so $\Psi(f) \in \Hom_{\ZG}(\Z[G/H]\otimes M, N)$. Hence, $\Psi$ is well-defined and it is apparent that $\Psi$ is a group homomorphism. A simple computation shows that $\Psi$ is a two-sided inverse of $\Phi$. 
\end{proof}

\begin{lemma}
	\label{LemmaIsomExt}
	Let $M$ and $N$ be left $\ZG$-modules. Let $\zgh \otimes M$ be equipped with the diagonal $G$-action. Then there are isomorphisms
	$$\Ext^r_{\ZG}(\zgh \otimes M,N) \cong \Ext^r_{\ZH}(\widetilde{M},\widetilde{N}) \qquad\forall r \in \NN_0.$$
\end{lemma}
\begin{proof}
	Let $$0 \to N \longhookrightarrow J_0 \stackrel{j_0}{\longrightarrow} J_1 \stackrel{j_1}{\longrightarrow} J_2 \stackrel{j_2}{\longrightarrow} \dots $$ be an injective resolution of $N$ over $\ZG$. By Lemma \ref{LemmaIsomHom}, there is an isomorphism 
	$$\Phi_i:\Hom_{\ZG}(\zgh \otimes M,J_i) \to \Hom_{\ZH}(\widetilde{M},\tilde{J}_i)$$
for each $i \in \NN_0$	and one checks without difficulties that $\Phi_i$ is compatible with the maps induced by the $j_i$. At this point, it suffices to show that each of the $\tilde{J}_i$ is an injective $\ZH$-module as passing to cohomology then shows the claim.
	
	Let $J$ be an injective $\ZG$-module, $X$ and $Y$ be $\ZH$-modules and $i: X \hookrightarrow Y$ be a monomorphism of $\ZH$-modules and let $f \in \Hom_{\ZH}(X,\tilde{J})$. 
	We consider the induced $\ZG$-modules $\Ind^G_H(X)$ and $\Ind^G_H(Y)$. One checks from the universal property of induced modules see e.g. \cite[p. 63]{Brown82}, that $i$ induces a $\ZG$-homomorphism $\widetilde{i}:\Ind^G_H(X) \hookrightarrow \Ind^G_H(Y)$, which is again injective since $\ZG$ is free as a right $\ZH$-module, and $f$ induces $\widetilde{f} \in \Hom_{\ZG}(\Ind^G_H(X),J)$. Since $J$ is injective over $\ZG$, it follows that there exists $\widetilde{\varphi} \in \Hom_{\ZG}(\Ind^G_H(Y),J)$ with $\widetilde{f} = \widetilde{\varphi} \circ \widetilde{i}$.
	
	Define $\varphi: Y \to J$ by $\varphi(y) := \widetilde\varphi(1 \otimes_{\ZH} y)$ for  each $y \in Y$. One checks without difficulties that $\varphi$ is a $\ZH$-homomorphism  with 
	$$(\varphi \circ i)(y) = \widetilde{\varphi}(1\otimes_{\ZH} i(y)) = \widetilde{\varphi}(\widetilde{i}(1 \otimes_{\ZH} y)) = \widetilde{f}(1 \otimes_{\ZH} y) = f(y)$$
	for all $y \in J$. This shows that $\tilde{J}$ is injective over $\ZH$ and thereby completes the proof.
\end{proof}
\begin{corollary}
\label{CorE0}
Let $r \in \NN$ and $s \in \NN_0$. Then 
$$E^{r,s}_0 \cong \Ext_{\ZH}(\widetilde{I}^s,\widetilde{A}).$$	
\end{corollary}
\begin{proof}
	This is the special case of Lemma \ref{LemmaIsomExt} obtained by letting $M= I^s$ and $N=A$.
\end{proof}


The following theorem summarizes the most important properties of the spectral sequence.

\begin{theorem}
	\label{TheoremFM}
	Let $n \in \NN$ and let $u \in H^n(G;A)$ with $u \neq 0$. 
	\begin{enumerate}[a)]
		\item The class $u$ is essential relative to $H$ if and only if $u \in D^{n,0}_n$.
		\item $D^{n,0}_1= \ker [\iota^*: H^n(G;I) \to H^n(H;\widetilde{I})]$, where $\iota^*$ is induced by the inclusion $\iota: H\hookrightarrow G$. 
		\item  Let $s\in \{0,1,\dots,n-1\}$. Then $u \in D^{n,0}_{s+1}$ if and only if 
		$$u \in D^{n,0}_{s} \qquad \text{and} \qquad u \in \ker \left[j_s:D^{n,0}_{s} \to E_{s}^{n-s,s}\right].$$
	\end{enumerate}
\end{theorem}
\begin{proof}
	\begin{enumerate}[a)]
	\item By Proposition \ref{PropPushfwdBerstein}, $u \in D^{n,0}_n$ if and only if there is a class $\mu \in H^0(G;\Hom_{\ZZ}(I^n,A))$, such that $u = \psi_*(\omega^n\cup \mu)$, where $\psi$ is described in the statement of said proposition. But $H^0(G;\Hom_{\ZZ}(I^n,A))= (\Hom_{\ZZ}(I^n,A))^G = \Hom_{\ZZ[G]}(I^n,A)$ and one checks without difficulties that, seeing $\mu$ as a $\ZG$-homomorphism, it holds that 
$$u= \psi_*(\omega^n\cup \mu) = \mu_*(\omega^n).$$
The claim immediately follows.
		\item By definition and exactness of the  exact couple, 
		\begin{align*}
			D^{n,0}_1&= \im \left[i_0: D^{n-1,1}_0 \to D^{n,0}_0 \right]= \ker \left[j_0:D^{n,0}_0 \to E^{n,0}_0 \right]\\
			&= \ker \left[j_0: \Ext^n_{\ZG}(\ZZ,A)\to \Ext^n_{\ZG}(\zgh,A) \right]\\
			&= \ker \left[j_0:  \Ext^n_{\ZG}(\ZZ,A) \to \Ext^n_{\ZH}(\ZZ,A)\right]\\
			&= \ker \left[\iota^*: H^n(G;A) \to H^n(H;A)\right],
		\end{align*}
		where we used Corollary \ref{CorE0}.
		\item This is an immediate consequence of the inner workings of exact couples. 
	\end{enumerate}
\end{proof}

\subsection{Computing the zero-th page} In \cite{FarMes20}, the authors proceeded from the construction of the spectral sequence by introducing certain decompositions of terms of the form $E^{r,s}_0$ for $r>0$ and $s >0$ into products of cohomology groups of centralizers of elements of the groups involved. We will show next that this can be generalized as well and derive decompositions of parts of our spectral sequence as products of cohomology groups of certain isotropy groups of $H$-actions that will be introduced momentarily. \medskip 

We consider the left $H$-action on the left cosets $G/H$ given by 
\begin{equation}
\label{EqHaction}
H \times G/H \to G/H, \qquad h \cdot gH = (hg)H.
\end{equation}
For each $s \in \NN$ we further consider the diagonal $H$-action
$$H \times (G/H)^s \to (G/H)^s, \quad h \cdot (g_1H,\dots,g_sH)=  (hg_1H,hg_2H,\dots,hg_sH).$$
We  denote the set of orbits of this action for each $s \in \NN$ by 
$$\C_s(G/H) := \{H \cdot (g_1H,g_2H,\dots,g_sH) \ | \ g_1H,\dots,g_sH \in G/H\}.$$
We put $(G/H)^*:= (G/H) \setminus \{H\}$ and 
$$\C_s'(G/H) := \{H \cdot (g_1H,g_2H,\dots,g_sH) \ | \ g_1H,\dots,g_sH \in (G/H)^*\}\subset \C_s(G/H).$$
The above action equips $\ZZ[G/H]^{\otimes s}$ with the structure of a left $\ZH$-module and we consider $I^s \subset \zgh^{\otimes s}$ as a $\ZH$-submodule. This submodule structure obviously coincides with the one obtained by $\widetilde{I}^s=(\res^G_H(I))^s$ that we previously considered.  One checks that as free abelian groups 
$$I^{ s} = \bigoplus_{g_1H,\dots,g_sH \in (G/H)^*} \ZZ\cdot  (g_1H-H)\otimes (g_2H-H) \otimes \dots \otimes (g_sH-H)\qquad \forall s \in \NN$$
and note that for all $s \in \NN$, $g_1H,\dots,g_sH \in G/H$ and $h \in H$ it holds that 
$$h \cdot (g_1H-H)\otimes (g_2H-H)\otimes \dots \otimes (g_sH-H)=(hg_1H-H)\otimes (hg_2H-H)\otimes \dots \otimes (hg_sH-H). $$
From this, one observes that for each $C \in \C'_s(G/H)$, we obtain a $\ZH$-submodule 
$$J_C \subset I^{ s}, \qquad J_C:= \bigoplus_{(g_1H,\dots,g_sH)\in C} \ZZ\cdot  (g_1H-H)\otimes (g_2H-H)\otimes \dots \otimes (g_sH-H), $$
and that 
\begin{equation}
	\label{EqItildedecomp}
	\widetilde{I}^{ s} = \bigoplus_{C \in \C'_s(G/H)} J_C
\end{equation}
is a decomposition of $\ZH$-modules. Moreover, for each $C \in \C'_s(G/H)$, we let $\ZC$ denote the free abelian group generated by the elements of $C$. One checks without difficulties that for each $C$ the map $\varphi_C: \ZC \to J_C$ that is obtained by $\ZZ$-linearly extending 
$$\varphi(g_1H,g_2H\dots,g_sH) = (g_1H-H)\otimes (g_2H-H)\otimes \dots \otimes (g_sH-H),$$
is an isomorphism of $\ZH$-modules. 

\begin{theorem}
\label{TheoremExtDecomp}
	Let $s \in \NN$. For each $C \in \C'_s(G/H)$ fix a representative $x_C \in C$ and let $N_C:= H_{x_C}$ be the isotropy group of $x_C$. Then 
	$$E^{r,s}_0 \cong \prod_{C \in \C'_s(G/H)} H^r(N_{C};\mathrm{Res}^G_{N_C}(A))\quad \forall r \in \NN.$$
\end{theorem}
\begin{proof}
	Fix $r \in \NN$. By Corollary \ref{CorE0}, it holds that $E^{r,s}_0 \cong \Ext^r_{\ZH}(\widetilde{I}^{ s},\widetilde{A})$. From this, using \eqref{EqItildedecomp} and the addivity of $\Ext$-functors we derive that
	$$E^{r,s}_0 \cong \prod_{C \in \C'_s(G/H)} \Ext^r_{\ZH}(J_C,\widetilde{A})\cong \prod_{C \in \C'_s(G/H)} \Ext^r_{\ZH}(\ZC,\widetilde{A}).$$
	Let $C \in \C'_s(G/H)$. For any left $\ZH$-module $A$ we observe that, since $H$ acts transitively on $C$, the map
	\begin{equation}
		\label{EqIsoZH}
		\Hom_{\ZH}(\ZC,A) \stackrel{\cong}{\to} (\Res^G_{N_C}(A))^{N_C}, \quad f \mapsto f(x_C), 
	\end{equation}
	is a group isomorphism. Let 
	$$\dots \longrightarrow P_r \stackrel{p_r}\longrightarrow P_{r-1} \stackrel{p_{r-1}}{\longrightarrow} \dots \stackrel{p_2}\longrightarrow P_1 \stackrel{p_1}\longrightarrow P_0 \stackrel{p_0}\longrightarrow \ZZ \longrightarrow 0$$
	be a free resolution of $\ZZ$ over $\ZH$. Since $\ZC$ is a free abelian group, it follows from \cite[Corollary III.5.7]{Brown82} that 
	$$\dots \longrightarrow \ZC \otimes P_r \stackrel{\id_{\ZC} \otimes p_r}{\longrightarrow} \ZC \otimes P_{r-1} \longrightarrow \dots \longrightarrow \ZC \otimes P_1 \stackrel{\id_{\ZC}\otimes p_1}{\longrightarrow} \ZC \otimes P_0 \longrightarrow \ZC \longrightarrow 0$$
is a free resolution of $\ZC$ over $\ZH$. Consequently, we can compute the above $\Ext$-groups as
	$$\Ext^r_{\ZH}(\ZC, \widetilde{A}) = H^r(\Hom_{\ZH}(\ZC\otimes P_*,A), (\id_{\ZC}\otimes p_*)^*)$$
	Let $r \in \NN_0$. If we consider $\Hom_{\ZZ}(P_r,\widetilde{A})$ as a left $\ZH$-module w.r.t. the diagonal $H$-action, then we obtain
	\begin{align*}
		\Hom_{\ZH}(\ZC\otimes P_r,A) &\cong \Hom_{\ZH}(\ZC,\Hom_{\ZZ}(P_r,A))\stackrel{\eqref{EqIsoZH}}{\cong}(\Hom_{\ZZ}(P_r,A))^{N_C}\\
		&= \Hom_{\ZZ[N_C]}(P_r,\Res^G_{N_C}(A))
	\end{align*}
	and one checks that an explicit isomorphism is given by
	$$F_r: \Hom_{\ZH}(\ZC\otimes P_r,\widetilde{A}) \to \Hom_{\ZZ[N_C]}(P_r,\Res^G_{N_C}(A)), \quad (F_r(f))(q)=f(x_C\otimes q) \quad \forall q \in P_r.$$
	One checks that the $F_r$ are compatible with the differentials, thus induce isomorphisms
	$$(F_r)_*: \Ext^r_{\ZH}(\ZC,\widetilde{A}) \to H^r(\Hom_{\ZZ[N_C]}(P_*,\Res^G_{N_C}(A)), p_r^*)  \qquad \forall r \in \NN_0,$$
	where we used the obvious fact that $\Res^H_{N_C}(\widetilde{A})=\Res^H_{N_C}(\Res^G_H(A))=\Res^G_{N_C}(A)$.
	Since each $P_r$ is free as a left $\ZH$-module, it is free as a left $\Z[N_C]$-module as well. Hence, $P_*$ is a free resolution of $\ZZ$ over $\ZZ[N_C]$, such that 
	$$ H^r(\Hom_{\ZZ[N_C]}(P_*,\Res^G_{N_C}(A)), p_r^*)= H^r(N_C;\Res^G_{N_C}(A)) \qquad \forall r \in \NN_0.$$
	Combining the previous observations shows the claim. 
\end{proof}


\subsection{Consequences for sectional categories of subgroup inclusions} 
We want to derive a lower bound on sectional categories of subgroup inclusions from Theorem \ref{TheoremExtDecomp}. We adopt all of the spectral sequence notation from the previous section.

\begin{definition}
Given a group $G$ and a subgroup $H \leq G$, we put
$$\kappa_{G,H} := \sup \{\cd(H \cap xHx^{-1}) \ |\ x \in G \setminus H\}.$$
	\end{definition}

The following result is a consequence of the previous properties of the spectral sequence:

\begin{proposition}
\label{PropositionDkappa}
Let $G$ be a geometrically finite group and let $H \leq G$ be a subgroup. Let $A$ be a left $\ZG$-module, let $n \in \NN$ and let $u \in H^n(G;A)=D^{n,0}_0$. If $n \geq \kappa_{G,H}$, then 
	$$ u \in D^{n,0}_{n-\kappa_{G,H}}.$$
\end{proposition}
\begin{proof}
For $x \in G$ we denote the isotropy group of $xH\in G/H$ with respect to the left $H$-action from \eqref{EqHaction} by $H_x$. Apparently, $H \cap xHx^{-1} \subset H_x$ for each $x \in G$. Conversely, let  $x \in G$, $g \in H_x$ and let $h_1 \in H$ be arbitrary. Then there exists an $h_2 \in H$ with 
	$$ g \cdot xh_1 = xh_2 \quad \Leftrightarrow \quad g = xh_2h_1^{-1}x^{-1} \quad \Rightarrow g \in xHx^{-1}.$$
Thus, we have shown that $H_x = H \cap xHx^{-1}$ for each $x \in G$.

Let $s \in \NN$. By definition of the $H$-actions, for each $C \in \C'_s(G/H)$ there is some $x \in G\setminus H$, such that $N_C \leq H_x$, which by \cite[Proposition VIII.2.4]{Brown82} yields that $$\cd(N_C) \leq \cd(H_x)= \cd(H \cap xHx^{-1}) \leq \kappa_{G,H}$$ for each $C \in \C'_s(G/H)$. In particular, $H^r(N_{C};\mathrm{Res}^G_{N_C}(A))=0$ whenever $r>\kappa_{G,H}$, so we derive from Theorem \ref{TheoremExtDecomp} that 
\begin{equation}
\label{EqProofKappa}
E^{r,s}_0 = \{0\} \qquad \forall r > \kappa_{G,H}, \ s \in \NN.
\end{equation}
In the case of $r=n-s$, we obtain
$$E^{n-s,s}_0 = \{0\} \qquad \forall s \in \{1,2,\dots,n-\kappa_{G,H}-1\}. $$
In terms of Theorem \ref{TheoremFM}.c), this yields that $u \in \ker[j_s:D^{n,0}_s\to E^{n-s,s}_s]$ for $1 \leq s \leq n-\kappa_{G,H}-1$, so it follows directly from Theorem \ref{TheoremFM}.c) that $u \in D^{n,0}_{n-\kappa_{G,H}}$. 
\end{proof}

This has an immediate consequence for sectional categories of subgroup inclusions.
\begin{theorem}
\label{TheoremSecatKappa}
	Let $G$ be a geometrically finite group and $H \leq G$ be a subgroup. Then 
	$$\secat(H \hookrightarrow G) \geq \cd(G)-\kappa_{G,H}.$$
\end{theorem} 
\begin{proof}
	Put $n := \cd(G)$, let $A$ be a left $\ZG$-module with $H^n(G;A) \neq \{0\}$ and let $u \in H^n(G;A)$ with $u \neq 0$. It follows from Proposition \ref{PropositionDkappa} that $u \in D^{n,0}_{n-\kappa_{G,H}}$. Thus, we obtain from Theorem \ref{TheoremSecatDp} that $\secat(H \hookrightarrow G) \geq n-\kappa_{G,H}$.
\end{proof}

\begin{remark}
	If $H$ is not a self-normalizing subgroup of $G$, then there exists some non-trivial element $x \in N_G(H) \setminus H$ such that $ H \cap xHx^{-1} = H $. Thus,  $\kappa_{G,H} = \cd(H)$ in this case and Theorem \ref{TheoremSecatKappa} yields $$ \secat(H \hookrightarrow G) \geq \cd(G) - \cd(H). $$ This is, of course, the case if we consider a normal subgroup $N \triangleleft G$. In this situation, with $Q=G/N$, and under the necessary assumption that $\cd(Q) < \infty$, we obtain $$ \cd(G) - \cd(N) \leq \secat(N \hookrightarrow G) \leq \cd(Q). $$ In Theorem \ref{TheoremSecatNormal} we have seen a condition for the sectional category to reach the top dimension. It is interesting to remark as well that if $H^{\cd(N)}(N;\ZZ[N])$ is free abelian then, by \cite[Theorem 5.5]{Bieri81}, we have that $\cd(Q) = \cd(G) - \cd(N)$ and thus $\secat(N \hookrightarrow G) = \cd(G) - \cd(N)$ under such assumption. 
\end{remark}

Before going on studying general sectional categories of subgroup inclusions, we next want to check explicitly that our Theorem \ref{TheoremSecatKappa} indeed recovers the corresponding result from \cite{FarMes20}.

\section{Applications to topological complexity}
\label{SectionAppl}
\subsection{Sequential topological complexity of aspherical spaces}

The notion of \emph{sequential} or \emph{higher topological complexities} was introduced by Y. Rudyak in \cite{RudyakHigher} as a generalization of topological complexity which models the motion planning problem for robots that are supposed to make some pre-determined intermediate stops along their ways. We briefly recall their definition.

Let $X$ be a path-connected topological space. For each $r \in \NN$ with $r \geq 2$ the map 
$$p_r: PX\to X^r, \qquad p_r(\gamma) = \left(\gamma(0),\gamma\big(\tfrac{1}{r-1}\big),\gamma\big(\tfrac{2}{r-1}\big),\dots,\gamma\big(\tfrac{r-2}{r-1}\big), \gamma(1)  \right),$$
is a fibration. The \emph{$r$-th sequential topological complexity of $X$} is defined as
$$\TC_r(X) := \secat(p_r:PX \to X^r).$$
Note that by definition $\TC_2(X)=\TC(X)$. The sequential topological complexities of aspherical spaces have been studied by M. Farber and J. Oprea in \cite{FOSequ}. In particular, given a geometrically finite group $\pi$, it is shown in \cite[Lemma 4.2 and Corollary 4.3]{FOSequ} that $\TC_r(K(\pi,1))$ coincides with the sectional category of the covering of $(K(\pi,1))^r$ that is associated with the diagonal subgroup 
$$\Delta_{\pi,r} := \{(g,g,\dots,g)\in \pi^r \ | \ g \in \pi\}.$$
In the light of Remark \ref{RemarkSecatInclCovering}, this shows that $\TC_r(K(\pi,1))$ is given as the sectional category of the inclusion of $\Delta_{\pi,r}$, i.e.
\begin{equation}
\label{EqTCrsecat}
\TC_r(K(\pi,1))= \secat(\Delta_{\pi,r} \hookrightarrow \pi^r).
\end{equation}
To obtain a tangible lower bound for $\TC_r(K(\pi,1))$ from Theorem \ref{TheoremSecatKappa}, we need to determine 
the value of $\kappa_{\pi^r,\Delta_{\pi,r}}$ more explicitly.

\begin{lemma}
\label{LemmaKappaHigher}
For each $r \in \NN$ with $r \geq 2$, it holds that
$$\kappa_{\pi^r,\Delta_{\pi,r}} = k(\pi) := \max \{\cd(C(g)) \ | \ g \in \pi \setminus \{1\}\},$$	
where $C(g)$ denotes the centralizer of $g \in \pi$.
\end{lemma}
\begin{proof}
Let $x=(x_1,\dots,x_r) \in \pi^r \setminus \Delta_{\pi,r}$ and let $(h,h,\dots,h) \in \Delta_{\pi,r}$. Then 
\begin{align*}
x(h,h,\dots,h)x^{-1} \in \Delta_{\pi,r} \quad &\Leftrightarrow \quad (x_1hx_1^{-1},x_2hx_2^{-1},\dots,x_rhx_r^{-1}) \in \Delta_{\pi,r} \\
&\Leftrightarrow \quad x_1hx_1^{-1}= x_2hx_2^{-1}= \dots =x_rhx_r^{-1}.
\end{align*}
For all $i,j \in \{1,2,\dots,r\}$ we compute that $$ x_ihx_i^{-1}=x_jhx_j^{-1} \quad \Leftrightarrow \qquad  x_j^{-1}x_ih=hx_j^{-1}x_i \quad \Leftrightarrow\quad h \in C(x_j^{-1}x_i). $$
One derives from this observation that 
$$ (h,h,\dots,h) \in \Delta_{\pi,r} \cap x\Delta_{\pi,r}x^{-1} \quad \Leftrightarrow \quad h \in \bigcap_{i \neq j} C(x_j^{-1}x_i).$$
This shows in particular that any subgroup of $\pi^r$ of the form $\Delta_{\pi,r} \cap x\Delta_{\pi,r}x^{-1}$, where $x \notin \Delta_{\pi,r}$, is isomorphic to a subgroup of the centralizer of an element of $\pi\setminus \{1\}$, so we derive that $\kappa_{\pi^r,\Delta_{\pi,r}} \leq k(\pi)$. On the other hand, given an arbitrary $g \in \pi$ with $g \neq 1$, if we put $x_0:= (g,1,\dots,1) \in \pi^r$, then it follows from the above that
$$(h,h,\dots,h) \in \Delta_{\pi,r} \cap x_0\Delta_{\pi,r}x_0^{-1} \quad \Leftrightarrow \quad h \in C(g),$$
so $C(g)$ is indeed isomorphic to a group of the form $\Delta_{\pi,r}\cap x\Delta_{\pi,r}x^{-1}$. This shows that $\kappa_{\pi,r} \geq k(\pi)$ and the two inequalities together show the claim. 
\end{proof}
Thus, we obtain the following consequence of our main lower bound.
\begin{theorem}
\label{TheoremLowerSeqTC}
	Let $\pi$ be a geometrically finite group and let $r \in \NN$ with $r \geq 2$. Then 
	$$\TC_r(K(\pi,1)) \geq r \cdot \cd(\pi)-k(\pi),$$
	where $k(\pi)= \max \{ \cd(C(g)) \ | \ g \in \pi\setminus \{1\}\}$.
\end{theorem}
\begin{proof}
	We derive from Theorem \ref{TheoremSecatKappa} and from \eqref{EqTCrsecat} that 
	\begin{align*}
	\TC_r(K(\pi,1)) \geq \cd(\pi^r) - \kappa_{\pi^r,\Delta_{\pi,r}} \\
	= r \cdot \cd(\pi) - k(\pi),		
	\end{align*}
where in the last step we used Lemma \ref{LemmaKappaHigher} and the main result of \cite{Dranish19} on cohomological dimensions of products of geometrically finite groups.
\end{proof}

\begin{remark}
\label{RemarkTCr}
\begin{enumerate}[(1)]
	\item In the case of $r=2$, the previous theorem yields the lower bound of 
	$$\TC(K(\pi,1))\geq 2\cd(\pi)-k(\pi)$$
	for topological complexity. Although not explicitly stated therein, this inequality is an immediate consequence of the main results of \cite{FarMes20}.
	\item If $\pi$ is a torsion-free hyperbolic group, then $C(g)$ is infinite cyclic for each $g \in \pi^*$, so Theorem \ref{TheoremLowerSeqTC} and Theorem \ref{TheoremSecatInclProperties}.a) imply that 
	$$r \cdot \cd(\pi)-1 \leq \TC_r(K(\pi,1)) \leq r \cdot \cd(\pi)  \qquad \forall r \geq 2$$
	in this case. It has in fact been shown by S. Hughes and K. Li in \cite{HugLi} that indeed $\tc_r(\pi)= r \, \cd(\pi)$ for all torsion-free hyperbolic groups with $\pi \not\cong \Z$ and all $r \geq 2$. However, the methods of \cite{HugLi} only generalize slightly beyond the hyperbolic case and do not yield a general lower bound for geometrically finite groups.
	\item If $\pi$ is a free abelian group, then $\Delta_{\pi,r}$ is a \emph{normal} subgroup of $\pi^r$ with $\pi^r/\Delta_{\pi,r} \cong \pi^{r-1}$. In this case, we derive from Theorem \ref{TheoremSecatNormal} that 
 $$\TC_r(\pi) = \cd(\pi^{r-1}) = (r-1) \cdot \cd(\pi) \qquad \forall r \geq 2.$$
This has already been observed in \cite[Corollary 3.13]{BGRT}.

\item Suppose that $x = (x_1, \cdots, x_r) \in \pi^r$ satisfies $x\Delta_{\pi,r}x^{-1} = \Delta_{\pi,r}$. From the proof of Lemma \ref{LemmaKappaHigher} we can infer that this implies $\pi \subset C(x_jx_i^{-1})$ for every $i \neq j$. But this means, in turn, that $x_j x_i^{-1} \in Z(\pi)$. Therefore $$ N_{\pi^r}(\Delta_{\pi,r}) = \{ (x_1, \cdots, x_r) \in \pi^r \ | \ x_j x_i^{-1} \in Z(\pi), \forall i \neq j  \}. $$ Consequently, if the group $\pi$ satisfies $Z(\pi) = \{ 1 \}$, the diagonal subgroup $\Delta_{\pi,r}$ is self-normalizing. 
	\end{enumerate}
\end{remark}

We want to apply Theorem \ref{TheoremLowerSeqTC} to a certain class of free amalgamated products whose centralizers were studied by T. Lewin. For this purpose, we need to introduce a notion from group theory.

\begin{definition}
	Let $G$ be a group. A subgroup $H \leq G$ is \emph{malnormal} if
	$$x H x^{-1} \cap H = \{1\} \qquad \forall x \in G \smallsetminus H.$$
\end{definition}

In the following, given a group $G$ and $g \in G$ we let $C_G(g)$ denote its centralizer whenever it is ambiguous which group we are referring to.

\begin{corollary}
\label{CorAmalg}
	Let $\pi_1$ and $\pi_2$ be geometrically finite groups and consider a free product with amalgamation $\pi_1 *_H \pi_2$, such that $H$ is malnormal in $\pi_1$ or malnormal in $\pi_2$.	Then for each $r \geq 2$ 
	$$\tc_r(\pi_1 *_H \pi_2) \geq r \cdot \cd(\pi_1 *_H \pi_2) - \max\{k(\pi_1),k(\pi_2)\}.$$
\end{corollary}
\begin{proof}
Put $\pi:= \pi_1 *_H \pi_2$ and	let $g \in \pi$, $g \neq 1$. By \cite[Theorem 2]{Lewin}, the centralizer $C_\pi(g)$ is infinite cyclic or isomorphic to $C_{\pi_1}(g)$ or $C_{\pi_2}(g)$. In the first case, it holds that $\cd(C_\pi(g))=1$, while in the other two cases it holds that $\cd(C_\pi(g))\leq k(\pi_1)$ or $\cd(C_\pi(g))\leq k(\pi_2)$, respectively. Since $g$ was chosen arbitrarily, this yields that
	$$k(\pi) \leq \max\{1,k(\pi_1) ,k(\pi_2) \}=\max\{k(\pi_1),k(\pi_2)\},$$
	so the claim follows immediately from Theorem \ref{TheoremLowerSeqTC}.
\end{proof}

\begin{remark}
	For explicit computations using Corollary \ref{CorAmalg}, there are some general results about cohomological dimensions of free amalgamated products that come in handy. More precisely, let $\pi_1$ and $\pi_2$ be groups of finite cohomological dimension and consider a free product with amalgamation $\pi_1 *_H \pi_2$. It is shown in \cite[Proposition 6.1]{Bieri81} that
$$\max\{\cd (\pi_1) ,\cd (\pi_2) \}\leq \cd (\pi_1*_H \pi_2) \leq \max\{\cd (\pi_1) ,\cd (\pi_2) \}+1.$$
and that a necessary condition for $\cd(\pi_1*_H\pi_2)= \max\{ \cd (\pi_1),\cd (\pi_2)\}+1$ to hold is that $\cd(\pi_1)=\cd(\pi_2)$. It is further shown in \cite[Corollary 6.5]{Bieri81} that a sufficient condition for this equality is that both $\pi_1$ and $\pi_2$ are of type $FP_\infty$ and that $H$ is of finite index both in $\pi_1$ and in $\pi_2$.
\end{remark}

\subsection{Parametrized topological complexity of epimorphisms} 
The parametrized topological complexity of a fibration has been introduced by D. Cohen, M. Farber and S. Weinberger in \cite{CFW}. Given a fibration $p:E \to B$, one considers $E^I_B$ as the space of all continuous paths $\gamma \colon I:=[0,1] \rightarrow E$ in a single fibre of $p$, i.e. such that the path $p \circ \gamma$ is constant. Define the space $$ E \times_B E = \{ (e,e') \in E \times E\  \vert \  p(e) = p(e') \} $$ of all possible pairs of configurations lying in the same fibre of $p$. Then, the map $$ \Pi : E^I_B \to E \times_B E \qquad \Pi(\gamma) = (\gamma(0), \gamma(1))  $$  is a fibration with fibre $\Omega X$. The parametrized topological complexity of $p$ is defined as
$$\TC[p:E\to B]= \secat(\Pi:E^I_B \to E \times_B E).$$
We want to apply the results previously obtained to the parametrized topological complexity of group epimorphisms. This algebraic variant of parametrized topological complexity was defined and investigated by M. Grant in \cite{Grant22}. Given two groups $G$ and $Q$ and an epimorphism $\rho:G \twoheadrightarrow Q$ there exists a fibration $f_{\rho}:K(G,1) \to K(Q,1)$, whose fibre is path-connected and which induced $\rho$ on the level of fundamental groups. Moreover, it is shown in loc. cit. that $\TC[f_\rho:K(G,1)\to K(Q,1)]$ is independent of the choice of $f_\rho$ and that 
$$\TC[f_\rho:K(G,1)\to K(Q,1)]= \secat(\Delta_G \hookrightarrow G \times_Q G) =: \TC[\rho:G \twoheadrightarrow Q].$$
Here, $\Delta_G= \{(g,g) \in G \times G\ | \ g \in G\}$ denotes the diagonal subgroup and 
$$G \times_Q G := \{(x,y) \in G \times G \ | \ \rho(x)=\rho(y)\}.$$
We discuss an alternative description of these pullback groups in the following lemma. 
\begin{lemma}
	\label{LemmaGQepi}
	Let $G,Q$ be groups, let $\rho: G \to Q$ be an epimorphism. Then 
		$$G \times_Q G = ((\ker \rho) \times 1)\cdot \Delta_G.$$	
\end{lemma}
\begin{proof}
Let $ k \in \ker \rho$ and $g \in G$. Then, since $\rho$ is a homomorphism, $\rho(kg) = \rho(k)\cdot \rho(g) = \rho(g)$, so that $(kg,g) \in G \times_Q G$. Conversely, let $(g_1,g_2) \in G \times_Q G$. Then
	$$\rho(x)=\rho(y) \quad \Leftrightarrow \rho(x)(\rho(y))^{-1}=1 \quad \Leftrightarrow \quad \rho(xy^{-1})=1 \quad\Leftrightarrow \quad xy^{-1} \in \ker \rho. $$
	Thus $(x,y) = (xy^{-1},1) \cdot (y,y) \in ((\ker \rho )\times 1)\cdot \Delta_G$.
\end{proof}

We now want to apply our results on sectional categories to this setting. The following statement is a straightforward application of Theorem \ref{TheoremSecatKappa}.
\begin{theorem}
	\label{TheoremTCepi}
	Let $G$ and $Q$ be geometrically finite groups and let $\rho: G \twoheadrightarrow Q$ be an epimorphism.
	Then 
	$$\TC[\rho:G \to Q]\geq \cd(G\times_Q G)-k(\rho), $$
	where	
	$$k(\rho)= \max \{\cd( C(g)) \ | \ g \in \ker \rho, \ g \neq 1 \}.$$
\end{theorem}
\begin{proof}
	It follows from Theorem \ref{TheoremSecatKappa} that $\TC[\rho:G \to Q]\geq \cd(G\times_Q G)-\ell$, where 
	$$\ell:=\kappa_{G\times_Q G,\Delta_G}= \max\{ \cd(\Delta_G \cap z\Delta_G z^{-1}) \ |\ z \in (G \times_Q G) \setminus \Delta_G\}.$$
	It only remains to show that $k(\rho) = \ell$. It follows from Lemma \ref{LemmaGQepi} that
	\begin{align*}
		\ell&= 	\max\{ \cd(\Delta_G \cap (xh,h)\Delta_G(xh,h)^{-1}) \ |\ x \in \ker \rho \setminus \{1\}, \ h \in G\} \\
		&=\max\{ \cd(\Delta_G \cap (x,1)\Delta_G(x,1)^{-1}) \ |\ x \in \ker \rho \setminus \{1\}\} ,
	\end{align*}
	since, evidently, $(h,h)\Delta_G(h,h)^{-1}= \Delta_G$ for all $h \in G$. 
	
	Let $x \in \ker \rho$ with $x \neq 1$ and let $g \in G$. Then 
\begin{align*}
&(g,g) \in (x,1)\Delta_G(x,1)^{-1} \quad \Leftrightarrow	\quad \exists h \in G: \ \ (g,g)=(x,1)(h,h)(x,1)^{-1}=(xhx^{-1},h) \\
&\Leftrightarrow \quad \exists h \in G: \ \ g=h \ \wedge \ g=xhx^{-1} \\
&\Leftrightarrow \quad g=xgx^{-1} \quad \Leftrightarrow \quad x=g^{-1}xg \quad \Leftrightarrow \quad g^{-1} \in C(x) \quad \Leftrightarrow \quad g \in C(x).
\end{align*}
	Moreover, the map $C(x)\to \Delta_G \cap (x,1)\Delta_G(x,1)^{-1}$, $g \mapsto (g,g)$, is easily seen to be a group isomorphism. We immediately derive that $k(\rho) = \ell$.
\end{proof}
To study the cohomological dimension of $G \times_QG$, we can characterize such pullback groups as semidirect products.
\begin{lemma}
	\label{LemmaSemidir}
	Let $G,Q$ be groups and let $\rho: G \twoheadrightarrow Q$ be an epimorphism. Then 
	$$\Phi: G \times_Q G \to (\ker \rho) \rtimes_\varphi G, \quad \Phi(g,h) = (gh^{-1},h),$$
	is a group isomorphism, where $\varphi: G \to \Aut(\ker \rho)$, $(\varphi(g))(x)= gxg^{-1}$.
\end{lemma}
\begin{proof}
	One checks without difficulties that $\Phi$ is injective. Moreover, for each $(x,y) \in (\ker \rho) \rtimes G$ it holds that $\Phi(xy,y)= (x,y)$, so that $\Phi$ is surjective as well. For all $(g_1,h_1),(g_2,h_2) \in G \times_Q G$ we further compute that
	\begin{align*}
		\Phi((g_1,h_1) \cdot (g_2,h_2))&= \Phi(g_1g_2,h_1h_2) = (g_1 g_2h_2^{-1}h_1^{-1},h_1h_2) \\
		&= (g_1h_1^{-1} \cdot h_1g_2h_2^{-1}h_1^{-1},h_1h_2) = (g_1h_1^{-1} (\varphi(h_1))(g_2h_2^{-1}),h_1h_2) \\
		&= (g_1h_1^{-1},h_1)\bullet (g_2h_2^{-1},h_2) = \Phi(g_1,h_1)\bullet \Phi(g_2,h_2).
	\end{align*}
	where $\bullet$ denotes multiplication in $(\ker \rho) \rtimes_\phi G$. Thus, $\Phi$ is an isomorphism. 
\end{proof}

\begin{corollary}
	Let $G$ and $Q$ be geometrically finite groups and let $\rho:G \twoheadrightarrow Q$ be an epimorphism. Then 
	$$TC[\rho: G \twoheadrightarrow Q] \leq \cd(G) + \cd(\ker \rho).$$
\end{corollary}
\begin{proof}
 It is well known that the cohomological dimension of a semidirect product is at most the sum of those of its factors. Thus, it follows from Lemma \ref{LemmaSemidir} and the lower bound from Theorem \ref{TheoremSecatInclProperties}.a) that
 $$\TC[\rho:G\twoheadrightarrow Q] \leq \cd(G \times_Q G ) \leq \cd(G) + \cd(\ker \rho).$$
\end{proof}

\begin{corollary}
\label{CorTCupper}
	Let $G$ and $Q$ be geometrically finite groups and let $\rho: G \twoheadrightarrow Q$ be an epimorphism. Assume that $H^n(\ker \rho, \ZZ[\ker \rho])$ is $\ZZ$-free for $n = \cd(\ker \rho)$. Then 
	$$2\cd(G)-\cd(Q)-k (\rho)\leq \TC[\rho:G\twoheadrightarrow Q] \leq 2\cd(G)-\cd(Q), $$
	where $k(\rho) = \max \{\cd(C(g)) \ | \ g \in \ker \rho, \, g \neq 1\}$.
\end{corollary}
\begin{proof}
Since $G$ is geometrically finite and $H^n(\ker \rho;\Z[\ker\rho])$ is $\Z$-free, it follows from \cite[Theorem 5.5]{Bieri81}, that 
\begin{equation}
\label{EqcdBieri}
\cd(\ker \rho) = \cd(G)-\cd(Q).
\end{equation}
The upper bound on $\TC[\rho:G \twoheadrightarrow Q]$ thus follows directly from  Corollary \ref{CorTCupper}. Regarding the lower bound, we derive from Lemma \ref{LemmaSemidir} and again \cite[Theorem 5.5]{Bieri81} that
$$\cd(G \times_Q G) = \cd(G)+ \cd(\ker \rho) \stackrel{\eqref{EqcdBieri}}{=} 2 \cd(G)-\cd(Q). $$
The lower bound is then an immediate consequence of Theorem \ref{TheoremTCepi}.
\end{proof} 

If we want to consider the case of the inclusion of a normal subgroup notice that, for $G$ and $Q$ groups and $\rho:G \twoheadrightarrow Q$ an epimorphism, then $\Delta_G$ is a normal subgroup of $G\times_Q G$ if and only if $\ker \rho \subset Z(G)$, where $Z(G)$ denotes the center of $G$. Indeed, since $(g,h) \in G \times_QG$, it holds that $\rho(g)=\rho(h)$ and thus $g^{-1}h \in \ker \rho$. Thus, if $\ker \rho \subset Z(G)$, this condition is satisfied for all $(g,h) \in G\times_QG$ and $x \in G$. Conversely, if $\Delta_G$ is normal $G\times_QG$, then we derive by taking $(g,h)=(a^{-1},1)$ for $a \in \ker \rho$ that indeed $\ker \rho \subset Z(G)$. As such, we are in the situation that the associated group extension $$ \{1\} \rightarrow  \ker(\rho) \rightarrow G \xrightarrow{\rho} Q \rightarrow \{1\} $$ is central. Therefore, by \cite[Corollary 5.2]{Grant22} we know that $$\TC[\rho: G\twoheadrightarrow Q] = \cd(\ker(\rho)).$$ 

But we can also derive an approach to this case as a consequence of the more general computation provided by Theorem \ref{TheoremSecatNormal}. 

\begin{proposition}
	Let $G$ and $Q$ be geometrically finite groups and let $\rho:G \twoheadrightarrow Q$ be an epimorphism. Assume that $\ker \rho$ lies in the center of $G$ and consider the homomorphism $$\phi: G\times_Q G \to \ker \rho, \qquad \phi(g,h)=gh^{-1}.$$ 
	Then
 $$\cd(\phi:G\times_QG \to \ker \rho) \leq \TC[\rho: G\twoheadrightarrow Q] = \cd(\ker \rho).$$
\end{proposition}
\begin{proof}
	We observe using Lemma \ref{LemmaGQepi} that the map $G\times_QG \to \ker \rho$, $(g,h) \mapsto gh^{-1}$, in fact induces a group isomorphism $\psi:(G\times_Q G)/\Delta_G \stackrel{\cong}{\to} \ker \rho$ by the assumption on $\ker \rho$. Since the projection $p:G \times_QG \to (G \times_Q G)/\Delta_G$ is easily seen to satisfy $\psi \circ p=\phi$, we derive from the assumptions, Theorem \ref{TheoremSecatNormal} and \cite[Corollary 5.2]{Grant22} that
	$$\TC[\rho:G \twoheadrightarrow Q] = \cd\left(\quot{(G \times_Q G)}{\Delta_G} \right)= \cd(\ker \rho)$$
	
	$$\TC[\rho:G \twoheadrightarrow Q] \geq \cd(p)=\cd(\phi).$$
\end{proof}


\section{Canonical classes for sequential topological complexity}

We want to revisit sequential topological complexity from a more topological point of view. Note that the Berstein-Schwarz classes that we used in Section 5.1 to derive results for sequential topological complexity were introduced and studied in a completely algebraic way. In this section, we want to introduce topologically defined analogues of Berstein-Schwarz classes for sequential topological complexity that are defined for all topological spaces, not only aspherical ones. The construction generalizes the notion of canonical classes that was introduced by A. Costa and the second author in \cite{Costa} for topological complexity. \medskip 

Throughout this section, let $\pi$ be a group. For each $r \in \NN$ with $r \geq 2$ we put 
$$\TC_r(\pi):= \TC_r(K(\pi,1)).$$

\subsection{The construction of canonical classes}   Let $r \geq 2$ be fixed and put $G:= \pi^r$. We consider the Cartesian power $\pi^{r-1}$ as a left $G$-space with respect to the action
\begin{equation}\label{raction}
( x_1, \dots, x_r) \cdot (g_1, \dots, g_{r-1}) = (x_1g_1x_2^{-1}, x_2g_2x_3^{-1}, \dots, x_{r-1}g_{r-1}x_r^{-1}),
\end{equation}
where $(x_1, \dots, x_r)\in G$ and $(g_1, \dots, g_{r-1})\in \pi^{r-1}$. We view $\Z[\pi^{r-1}]$ as a left $\ZG$-module with respect to this action and consider its augmentation ideal 
$$I_r := \ker \left[\varepsilon: \Z[\pi^{r-1}]\to \ZZ\right]$$
again as a left $\ZG$-module. \medskip 

We consider the map 
\begin{equation}
\label{Eqfr}
f_r: \pi^r \to I_r, \quad f_r(g_1,g_2,\dots,g_r)= (g_1g_2^{-1}-1,g_2g_3^{-1}-1,\dots,g_{r-1}g_r^{-1}-1).
\end{equation}
One checks without difficulties that $f_r$ is a \emph{crossed homomorphism} and as such defines a cohomology class on $X^r$, see \cite[Section VI.3]{Whitehead}.

\begin{definition}
	Let $X$ be a path-connected topological space with $\pi_1(X)\cong \pi$ and let $r \in \NN$ with $r \geq 2$. The \emph{$r$-th canonical class of $X$} is the cohomology class
	$$\vv_r \in H^1(X^r;I_r), \quad \vv_r:=[f_r],$$
	i.e. the class induced by the crossed homomorphism $f_r$.
\end{definition}
Note that this definition generalizes the one from \cite[Section 2]{Costa} dealing with the case of $r=2$.\medskip

We want to describe the classes $\vv_r$ for cell complexes from a more topological point of view. Throughout the following, let $X$ be a cell complex with $\pi_1(X) = \pi$ and let $r\geq 2$ be fixed. Here, we shall suppress the chosen basepoint from the notation. We further let $H_*(Y)$ denote the singular homology of a space $Y$ with integer coefficients. \enskip   

The fibre of the free path fibration $p_r$ of $X$ is $(\Omega X)^{r-1}$, where $\Omega X$ denote the based loop space of $X$.  The action of $G \cong \pi_1(X^r)$ on the fibre of $p_r$ induces a $G$-action on $H_0((\Omega X)^{r-1})$. We view $H_0((\Omega X)^{r-1})$ and the reduced homology group $\widetilde{H}_0((\Omega X)^{r-1})$ as left $\ZG$-modules with respect this action. 

 By \cite[Theorem 1]{Schwarz66}, there is a so-called 
 \emph{homological obstruction} to the existence of a continuous section of $p_r$ over the $1$-skeleton of $X^r$, i.e. a cohomology class 
 $$\theta \in H^1(X^r; \widetilde H_0((\Omega X)^{r-1}))$$ with the property that $p_r$ admits a section over the $1$-skeleton of $X^r$ if and only if $\theta =0$. 
Considering a fixed isomorphism $\pi\cong \pi_1(X)$, we obtain a bijection 
 $$F: \pi_0((\Omega X)^{r-1}) \to \pi_1(X^{r-1})\cong \pi^{r-1}.$$
\begin{lemma}
	The group homomorphism $\Phi: H_0((\Omega X)^{r-1})\to \Z[\pi^{r-1}]$ that is induced by the bijection $F:\pi_0(\Omega X)^{r-1}\to \pi^{r-1}$, is an isomorphism of $\ZG$-modules. 
\end{lemma}
\begin{proof}
	It is evident that $\Phi$ is a group isomorphism, so it only remains to show its compatibility with the $G$-actions. For this purpose, we need to study the $G$-action on $H_0((\Omega X)^{r-1})$ in greater detail. 
	
Given a fixed fibre of $p_r$ and a loop $\sigma \in \Omega(X^r)$, $\sigma(t)=(\alpha_1(t), \ldots, \alpha_r(t))$ in the base space of $p_r$, we obtain its monodromy or holonomy map 
$$K_\sigma \colon \Omega X^{r-1} \to \Omega X^{r-1}$$
associated to the fibration $p_r$. This monodromy is the $(r-1)$-componentwise version of the one
described in \cite{Costa} and is up to homotopy given as
\begin{equation}
\label{EqMonodrom}
K_\sigma(\omega)\simeq(\alpha_1\omega_1\bar\alpha_2, \alpha_2\omega_2 \bar\alpha_3,\ldots,\alpha_{r-1}\omega_{r-1}\bar\alpha_r)
\end{equation}
where $\bar\alpha$ denotes the inverse path to $\alpha$ and $\omega=(\omega_1,\ldots,\omega_{r-1})$ and where for any two loops $\alpha$ and $\beta$ we let $\alpha\beta$ denote their concatenation.  Let $g=(g_1,\dots,g_r) \in G$ and $[\omega] \in \pi_0((\Omega X)^{r-1})$. The $G$-action $g \cdot [\omega]$ is then given by $\Z$-linearly extending the following construction: consider a loop $\sigma=(\alpha_1,\dots,\alpha_r) \in \Omega(X^r)$ with $[\sigma]=g$ and let
$$g \cdot [\omega] := [K_\sigma(\omega)].$$
Since, apparently, $[\alpha_i]=g_i$ for each $i \in \{1,2,\dots,r\}$, one checks from this description without difficulties that $\Phi$ is indeed a $G$-map with respect to this action and the one described in \eqref{raction}.
\end{proof}
One checks that $\Phi$ restricts to an isomorphism of $\ZG$-modules
$$\varphi: \widetilde{H}_0((\Omega X)^{r-1}) \to I_r.$$

\begin{proposition}
\label{PropObstr}
	The canonical class and the homological obstruction class are related by 
	$$\varphi_*(\theta) = \vv_r.$$
\end{proposition}
\begin{proof}
Assume throughout the following that $X$ has a unique $0$-cell $x_0$. Consider the universal covering $\widetilde{X}\to X$ equipped with the induced cell complex structure and consider its $r$-fold power $\widetilde{X}^r \to X^r$ as a universal covering for $X^r$. By construction of $\theta$, a cocycle $$c \in C^1(X^r;\widetilde{H}_0((\Omega X)^{r-1}))=\Hom_{\ZG}\left(C^{\text{cell}}_1(\widetilde{X}^r), \widetilde{H}_0((\Omega X)^{r-1})\right)$$ which represents $\theta$ is obtained as follows:

Let $\omega_0 \in (\Omega X)^{r-1}$ be given such that each component of $\omega_0$ is the constant loop in $x_0$. Given a $1$-cell $e$ of $X^r$ and a path $\gamma_e:[0,1] \to \widetilde{X}^r$ parametrizing $e$, we consider a fixed $1$-cell $\widetilde{e}$ of $\widetilde{X}^r$ which lifts $e$ and put
$$c(\widetilde{e})= [K_{\gamma_e}(\omega_0)]-[\omega_0]. $$
This expression is extended $\ZG$-equivariantly to the free $\ZG$-module $C^{\text{cell}}_1(\widetilde{X}^r)$.

If $e $ is an oriented $1$-cell of $X$ and a fixed lift $\widetilde{e}$ of $e$ as an oriented $1$-cell of $X$, then we have $1$-cells in $X^r$ given by
$$e_1=(e,x_0,\ldots,x_0),\ e_2=(x_0,e,x_0,\ldots,x_0), \ \ldots,\ e_r=(x_0,\ldots,x_0,e),$$
and $1$-cells in $\widetilde{X}^r$ defined analogously and denoted by $\widetilde{e}_1,\widetilde{e}_2,\dots,\widetilde{e}_r$. Given a path $\gamma_e:[0,1]\to X$ which parametrizes $e$, we define paths 
$$\gamma_j:[0,1]\to X^r, \qquad\gamma_j(t) = (x_0,\dots,x_0,\gamma_e(t),x_0,\dots,x_0). $$
where $\gamma_e(t)$ occurs in the $j$-th component of $\gamma_j$ for each $j \in \{1,2,\dots,r\}$. Evidently, $\gamma_j$ parametrizes $e_j$ for each $j$. We observe from \eqref{EqMonodrom} that the monodromy map of $p_r$ is up to homotopy given by 
$$K_{e_j}(\omega_1,\dots,\omega_{r-1}) \simeq (\omega_1,\ldots,\omega_{j-1}\bar \gamma_e,\gamma_e\omega_j,\omega_{j+1},\ldots,\omega_{r-1}). $$ 
Choosing the $\widetilde{e}_j$ as the distinguished lifts of the $e_j$ in the definition of the cocycle $c$, we thus obtain that 
\begin{align*}
c(\widetilde{e}_j) &= [K_{\gamma_j}(\omega_0)]-[\omega_0] \\
&= [(x_0,\dots,x_{0}\bar\gamma_e,\gamma_ex_0,x_0,\dots,x_0)]-[(x_0,x_0,\dots,x_0)]
\end{align*}
where we denoted the constant loop at $x_0$ simply by $x_0$. Here, $\gamma_ex_0$ occurs in the $j$-th component for each $j \in \{1,2,\dots,r-1\}$ and $x_0\bar\gamma_e$ occurs in the $(r-1)$-th component for $j=r$.

Since $c$ represents $\theta$, the class $\varphi_*(\theta)$ is represented by $ c':=\varphi \circ c \in \Hom_{\ZG}(C_1^{\text{cell}}(\widetilde{X}^r),I_r)$. From the definition of $\varphi$ and our computation of $c$, we observe with $g=[\gamma_e] \in \pi$ that
$$c'(\widetilde{e}_j) = (1,\dots,1,g^{-1},g,1,\dots,1)- (1,1,\dots,1)= (0,\dots,0,g^{-1}-1,g-1,0,\dots,0),$$
where $g-1$ occurs in the $j$-th component for each $j\in \{1,2,\dots,r-1\}$ and $g^{-1}-1$ occurs in the $r$-th component for $j=r$. Following the methods carried out in \cite[Section III]{Whitehead}, we can use $c'$ to construct a crossed homomorphism $k: G \to I_r$ which represents $\varphi_*(\theta)$ and obtain that 
$$k(1,\dots,1,g_j,1,\dots,1)= (0,\dots,0,g_j^{-1}-1,g_j-1,0,\dots,0),$$
for $g_j \in \pi$ and all $j \in \{1,2,\dots,r\}$. Using the crossed homomorphism property, we compute from this equation that 
\begin{align*}
k(g_1,\ldots,g_r)&=k(g_1,1,\ldots,1) +\sum_{i=1}^{r-1} (g_1,1,\ldots,1)\cdots (1,\ldots,1,g_i,1,\ldots,1)k(1,\ldots,1,g_{i+1},1,\ldots,1) \\
&= (g_1-1,0,\ldots,0) +\sum_{i=1}^{r-2} (0,\ldots, 0,g_ig_{i+1}^{-1}-1-(g_i-1),g_{i+1}-1,0,\ldots,0) \\
&\qquad \qquad +(0,\ldots,0,g_{r-1}g_r^{-1}-1-(g_r-1)) \\
&=  (g_1g_2^{-1}-1,g_2g_3^{-1}-1,\ldots, g_{r-1}g_r^{-1}-1) \\
&= f_r(g_1,\ldots,g_r).
\end{align*}
Therefore, the homological obstruction obeys $\varphi_*(\theta)=[f_r]=\vv_r$.
\end{proof}

\subsection{A homotopical viewpoint} There is an alternative ``homotopical obstruction'' viewpoint which arises from the fact that measuring the difference in connected
components, which the homological obstruction class does, may be accomplished by using the connecting homomorphism
$$\partial\colon \pi_1(X^r) \to \pi_0((\Omega X)^{r-1})$$
in the exact homotopy sequence associated with the free path fibration $p_r:PX \to X^r$.
Note that $\partial$ arises
from applying $\pi_0$ to a composition of inclusion with monodromy
$$(\Omega X)^r \hookrightarrow (\Omega X)^r \times (\Omega X)^{r-1} \to (\Omega X)^{r-1}.$$
Using the above identifications, we consider $\partial$ as a map $\partial\colon G \to \pi^{r-1}$ and by a standard property of the connecting homomorphism, it is a $G$-map with
$\partial(1)=1$, i.e. the components of the constant loops. To see the difference between connected components of the fibre, we define
$$\hat{k}: G \to I_r, \quad \hat{k}(g) = \partial(g)-1.$$
Then $\hat{k}$ is a crossed homomorphism as we see by
\begin{align*}
\hat{k}(gh)&= \partial(gh)-1 = g\partial(h)-\partial(g) + \partial(g) -1 \\
&= g(\partial(h)-1) + \partial(g)-1 \\
&= g\hat{k}(h) + \hat{k}(g).
\end{align*}
In fact, $\partial$ was computed in \cite[Proposition 2.1]{GLO} (with an opposite sign convention) to be
$$\partial(g_1,\ldots,g_r)= (g_1g_2^{-1},g_2g_3^{-1},\ldots, g_{r-1}g_r^{-1}),$$
from which we see we see that $\hat{k} = k$, where $k$ is the crossed homomorphism from the proof of Proposition \ref{PropObstr}. Consequently, the homological obstruction is $\theta = [\hat{k}]$.

\subsection{Naturality of canonical classes}  
Let $r \geq 2$, let $X$ and $Y$ be path-connected CW complexes and denote the path fibrations as
	$$p^X_r:PX \to X^r, \qquad p^Y_r: PY \to Y^r.$$
	Assume w.l.o.g. that $X$ has a unique $0$-cell $x_0$ and $Y$ has a unique $0$-cell $y_0$ and consider them as basepoints throughout this subsection. Let $f:X \to Y$ be continuous with $f(x_0)=y_0$ and consider the map 
	$$f_\#: (\Omega X)^{r-1} \to (\Omega Y)^{r-1}, \qquad f_\#(\alpha_1,\alpha_2,\dots,\alpha_{r-1}) := (f \circ \alpha_1, f\circ \alpha_2,\dots,f \circ \alpha_{r-1}),$$
where the based loop spaces are considered with the basepoints $x_0$ and $y_0$.  Apparently, $f_\#(\omega_{x_0})=\omega_{y_0}$, where $\omega_{x_0}$ and $\omega_{y_0}$ are in each component given by the constant loops at $x_0$ and $y_0$, respectively. Moreover, for suitable inclusions of the fibers, one checks that $f_\#$ coincides with the restriction of the map $PX \to PY$, $\gamma \mapsto f \circ \gamma$, to the fiber of $p_r$ over $x_0$.
	We put 
 $$A_X := \widetilde{H}_0((\Omega X)^{r-1}), \qquad A_Y:= \widetilde{H}_0((\Omega Y)^{r-1})$$
and denote  the map induced by $f_\#$ between the reduced homology groups by
 $$\varphi_* := \widetilde{H}_0(f_\#):A_X \to A_Y.$$
Let $f^*A_Y$ denote the $\pi_1(X^r)$-module which coincides with $A_Y$ as a free abelian group and whose $\pi_1(X^r)$-action is obtained from the $\pi_1(Y^r)$ action via $(f^r)_*$.

\begin{proposition}
Let $\theta_X \in H^1(X^r;A_X)$ and $\theta_Y\in H^1(Y^r;A_Y)$ be the homological obstructions of $p_r^X$ and $p_r^Y$, respectively.	Assume that $\pi_1(f):\pi_1(X,x_0) \to \pi_1(Y,y_0)$ is an isomorphism. Then
$$\varphi_*(\theta_X) = (f^r)^*(\theta_Y),$$
where $f^r:X^r \to Y^r$ is the $r$-fold Cartesian product of $f$ with itself.
\end{proposition}
\begin{proof}

Since $\pi_1(f)$ is an isomorphism, $\pi_0(f_\#): \pi_0(\Omega_r X) \to \pi_0(\Omega_r Y)$
is a bijection and thus $\varphi$ is an isomorphism of $\ZG$-modules.  Given $\alpha \in \Omega X^r$ and $\alpha' \in \Omega Y^r$, we let 
 $$K^X_\alpha: (\Omega X)^{r-1} \to (\Omega X)^{r-1}, \qquad K^Y_{\alpha'}: (\Omega Y)^{r-1} \to (\Omega Y)^{r-1},$$
 denote the monodromy maps of $p_r^X$ and $p_r^Y$, respectively. Using  \eqref{EqMonodrom}, we observe for each $\alpha = (\alpha_1,\dots,\alpha_r) \in \Omega X^r$ and all $\omega=(\omega_1,\dots,\omega_{r-1}) \in (\Omega X)^{r-1}$ that
\begin{align*}
&(f_\# \circ K^X_\alpha)(\omega) \simeq f_\#(\alpha_1\omega_1\bar\alpha_2, \alpha_2\omega_2 \bar\alpha_3,\ldots,\alpha_{r-1}\omega_{r-1}\bar\alpha_r) \\
	&= (f\circ (\alpha_1\omega_1\bar\alpha_2), f\circ (\alpha_2\omega_2 \bar\alpha_3),\ldots,f_\circ (\alpha_{r-1}\omega_{r-1}\bar\alpha_r)) \\
&= ((f\circ \alpha_1)(f\circ \omega_1)(\overline{f\circ \alpha_2}), (f\circ \alpha_2)(f\circ \omega_2)(\overline{f\circ \alpha_3}),\ldots,(f\circ \alpha_{r-1})(f\circ \omega_{r-1})(\overline{f\circ \alpha_r})) \\
&\simeq K^Y_{(f \circ \alpha_1,\dots,f\circ \alpha_{r})}(f\circ \omega_1,f \circ \omega_2,\dots,f \circ \omega_{r-1})= K^Y_{f^r\circ \alpha}(f_\#(\omega)) ,
\end{align*}
so that 
\begin{equation}
\label{EqMonodComm}
f_\# \circ K^X_\alpha \simeq K^Y_{f^r \circ \alpha} \circ f_\# \qquad \forall \alpha \in \Omega X^r.
\end{equation}
	By the cellular approximation theorem, we can assume w.l.o.g. that $f$ is a cellular map.  Let $Z(X^r)$ and $Z(Y^r)$ be the sets of those $1$-cells of $X^r$ and $Y^r$, respectively, whose homology classes are non-trivial. Since $\pi_1(f)$ is an isomorphism, $\pi_1(f^r):\pi_1(X^r) \to \pi_1(Y^r)$ is an isomorphism as well and  $f^r$ induces a bijection $Z(X^r) \to Z(Y^r)$ that we shall denote by $f$ as well.   
	
	Let $\widetilde{X}$ and $\widetilde{Y}$ be the universal covers of $X$ and $Y$, respectively, and let $\widetilde{f}:\widetilde{X} \to \widetilde{Y}$ be a lift of $f$. For each $e \in Z(X^r)$ we choose and fix a lift $\widetilde{e}$ to $\widetilde{X}^r$, i.e. a $1$-cell of $\widetilde{X}^r$ which projects down to $e$ under the universal covering map. Then for each $d \in Z_1(Y^r)$, there is a unique $e \in Z(X^r)$ with $f^r(e)=d$ and  the $1$-cell $\widetilde{f}^r(\widetilde{e})$ of $\widetilde{Y}^r$ lifts $d$. We equip each $d \in Z(Y^r)$ with the thus-obtained lift to $\widetilde{Y}^r$. As in the proof of Proposition \ref{PropObstr}, we define cocycles $c_X:C^{\text{cell}}_1(\widetilde{X}^r) \to A_X$ and $c_Y:C^{\text{cell}}_1(\widetilde{Y}^r) \to A_Y$ representing $\theta_X$ and $\theta_Y$ and defined with respect to the chosen lifts as in the proof of Proposition \ref{PropObstr}. 	Then  $(f^r)^*\theta_Y$ is represented by $(f^r)^*c_Y$, for which we compute that 
	\begin{align*}
		((f^r)^*c_Y)(\widetilde{e}) &= c_Y(\widetilde{f}^r(\widetilde{e})) \\
		&=[K^Y_{(f^r)\circ \gamma_e}(\omega_{y_0})]-[\omega_{y_0}] \\
		&= [K^Y_{(f^r)\circ \gamma_e}(f_\#(\omega_{x_0})]-[f_\#(\omega_{x_0})] \\ 
		&\stackrel{\eqref{EqMonodComm}}{=} [(f_\#\circ K^X_{\gamma_e})(\omega_{x_0})-[f_\#(\omega_{x_0})] \\
		&= (f_\#)_*\big([K^X_{\gamma_e}(\omega_{x_0})]-[\omega_{x_0}]\big) \\
		&= (f_\#)_* ( c_X(\widetilde{e})).
	\end{align*}
	Passing to cohomology shows the claim.
\end{proof}

We have seen that up to identifications of coefficient modules the homological obstruction classes considered above coincide with the respective $r$-th canonical classes. Thus, if we neglect some technical details, we immediately obtain the following statement.

\begin{corollary}
\label{CorCanonicalNatural}
Let $r \geq 2$ and let $\vv^X_r \in H^1(X^r;I_r)$ and $\vv^Y_r \in H^1(Y^r;I_r)$ be the $r$-th canonical classes of $X$ and $Y$, respectively. If $\pi_1(f):\pi_1(X)\to \pi_1(Y)$ is an isomorphism, then, up to a suitable isomorphism of coefficient modules, 
$$(f^r)^*(\vv^Y_r) = \vv^X_r.$$	
\end{corollary}

This yields an interesting connection between canonical classes of arbitrary CW complexes with nontrivial fundamental groups and canonical classes of aspherical spaces.

\begin{corollary}
\label{CorCanonicalPull}
Let $X$ be a connected CW complex with $\pi_1(X)=\pi$, let $K$ be a space of type $K(\pi,1)$ and let $f_X:X \to K$ be a classifying map for the universal cover of $X$. Then 
$$\vv^X_r = (f_X^r)^*(\vv^K_{r}) \in H^1(X^r,I_r).$$
\end{corollary}
\begin{proof}
 This is an immediate consequence of Corollary \ref{CorCanonicalNatural}, since by definition of $\kappa$, the map $\pi_1(\kappa):\pi_1(X) \to \pi$ is an isomorphism. 
\end{proof}

\subsection{Canonical classes of aspherical spaces}

Let $\pi$ be a group, fix $r \in \NN$ with $r\geq 2$ and put $G:= \pi^r$. We denote the group cohomology class that is induced by the crossed homomorphism $f_r$ of \eqref{Eqfr} by 
$$\vv_{r,\pi} \in H^1(G;I_r),$$
which can be seen as the $r$-th canonical class of a space of type $K(\pi,1)$. As a first step in connecting canonical classes to the aspects from the first part of this manuscript,  we want to relate this purely algebraically defined class to relative Berstein-Schwarz classes. 

We consider $G/\Delta_r$ as a left $G$-set in the obvious way and again consider $\pi^{r-1}$ as equipped with the $G$-action described in \eqref{raction}.
Let $\phi: \pi^r\to \pi^{r-1}$ be given by
\begin{eqnarray}\label{phi}
\phi(x_1, x_2, \dots, x_r)=(x_1x_2^{-1}, x_2x_3^{-1}, \dots, x_{r-1}x_r^{-1}). 
\end{eqnarray}
It is easy to see that $\phi$ is $G$-equivariant and descends to a $G$-equivariant bijection 
$$\bar\phi: G/\Delta_r \to \pi^{r-1},$$
which in turn induces a $\ZG$-module isomorphism 
$$\psi: \Z[G/\Delta_r] \to \Z[\pi^{r-1}].$$
Let $\sigma: \Z[G/\Delta_r] \to \ZZ$ be the augmentation and let $J:=\ker \sigma$. We recall that $J$ is the cofficient module from which the Berstein-Schwarz class of $G$ relative to $\Delta_r$ is obtained.

\begin{proposition}
\label{PropCanonBerstein}
Let $\omega \in H^1(G;J)$ be the Berstein-Schwarz class of $G$ relative to $\Delta_r$. Then 
$$\psi_*(\omega)=\vv_{r,\pi}. $$
\end{proposition}
\begin{proof}
It is apparent from its construction in \cite{BCE22} that $\omega$ is induced by the crossed homomorphism 
$$w: G \to J, \qquad w(g) = g\Delta -\Delta.$$
Thus, $\psi_*(\omega)$ is induced by $\psi \circ w$ and we compute for all $g_1,\dots,g_r \in \pi$ that 
\begin{align*}
	(\psi \circ w)(g_1,\dots,g_r)&= \psi((g_1,\dots,g_r)\Delta)-\psi(\Delta) \\
	&= (g_1g_2^{-1},g_2g_3^{-1},\dots,g_{r-1}g_r^{-1})- (1,1,\dots,1) \\
	&= (g_1g_2^{-1}-1,g_2g_3^{-1}-1,\dots,g_{r-1}g_r^{-1}-1)= f_r(g_1,\dots,g_r).
\end{align*}
This shows that $\psi_*(\omega) = [\psi \circ w] = [f_r]=\vv_{r,\pi}$.
\end{proof}
Thus, up to an isomorphism of coefficient modules, $\omega$ and $\vv_{r,\pi}$ indeed coincide.  

\begin{corollary}
\label{CorSeqTCCanonical}
Let $r \in \NN$ with $r\geq 2$.
\begin{enumerate}[a)]
	\item Then
$$\TC_r(\pi) \geq \height(\vv_r) = \sup \{n \in \NN \ |\ \vv_r^n\neq 0\}.$$	
\item Assume that $\pi$ is geometrically finite and that $\pi$ is not free or $r \geq 3$. Then $\TC_r(\pi) = r \cdot \cd(\pi)$ if and only if $$\height(\vv_r)=r\cdot \cd(\pi).$$
\end{enumerate}
\end{corollary}	
\begin{proof}
\begin{enumerate}[a)]
\item Let $k \in \NN$. We derive from Proposition \ref{PropCanonBerstein} and the compatibility of push-forwards with cup products that
$$\vv_r^k = \vv_{r,\pi}^k = (\psi_*(\omega))^k = (\psi^{\otimes k})_*(\omega^k).$$
Thus, the claim is a straightforward application of Proposition \ref{PropSecatOmega} to the case of $G=\pi^r$ and $H=\Delta_r$.
\item We derive from Theorem \ref{TheoremSecatInclProperties}.a) and the main result of \cite{Dranish19} that 
 $$\TC_r(\pi) \leq \cd(\pi^r)=r \cdot \cd(\pi).$$
 Thus, if $\height(\vv_r)=r \cdot \cd(\pi)$, it follows immediately from a) that $\TC_r(\pi)=r \cdot \cd(\pi)$. 
 
 Conversely, if $\TC_r(\pi)=\secat(\Delta_r \hookrightarrow \pi^r)=r \cdot \cd(\pi)$, it follows from \cite[Theorem 2.5]{BCE22} that $\omega^{r\cdot \cd(\pi)}\neq 0$, where $\omega$ denotes the Berstein-Schwarz class of $\pi^r$ relative to $\Delta_r$. Given that $$\psi^{\otimes (r\cdot \cd(\pi))}_*(\omega^{r \cdot \cd(\pi)})= (\psi_*(\omega))^{r \cdot \cd(\pi)}= \vv_r^{r\cdot  \cd(\pi)}$$ and since $\psi_*$ is an isomorphism of $\ZG$-modules, it follows that $\vv_r^{r\cdot  \cd (\pi)}\neq 0$ and thus $\height(\vv_r) \geq r\cdot  \cd (\pi)$, which becomes an equality, as for degree reasons $\height(\vv_r)\leq \cd(\pi^r)=r \cdot \cd(\pi)$.
\end{enumerate}
\end{proof}

\section{Topological and algebraic approaches to sequential TCs}

In this section, we want to combine our knowledge on canonical classes with the results from Section \ref{SectionAppl} for aspherical spaces to derive results on sequential topological complexities of spaces that are not necessarily aspherical. 

\subsection{Sequential $\D$-topological complexity}
As a first step, we recall some topological characterizations of $\tc_r(\pi)$ that have been carried out by the second and fourth authors in \cite{FOSequ}. We will briefly recall their constructions and results. Throughout this subsection, we consider a given group $\pi$ and its Cartesian powers $\pi^r$ and $\pi^{r-1}$ as discrete topological spaces. \medskip

Let $\D$ denote the family of subgroups of $\pi^r$ generated by the diagonal subgroup 
$$\Delta_r= \{(g,g,\dots,g) \in \pi^r \ | \ g \in \pi\},$$ that is, the smallest set of subgroups of
$G$ that contains $\Delta_r$ and is closed under conjugation and finite intersection.

\begin{definition}[{\cite[Definition 4.1]{FOSequ}}]\label{rdef2}
Let $X$ be a path-connected topological space with fundamental group $\pi$. {\it The $r$-th $\D$-topological complexity}, $\tc^\D_r(X)$, is defined as the minimal number
$k$ such that $X^r$ can be covered by $k+1$ open subsets
$$X^r = U_0 \cup U_1 \cup \dots U_k$$ with the property that for any $i \in \{0, 1,\dots,k\}$
and for any choice of the base point $u_i\in U_i$ the homomorphism $\pi_1(U_i, u_i)\to \pi_1(X^r, u_i)$ induced by
the inclusion $U_i\to X^r$ takes values in a subgroup of $\pi^r$ that is conjugate to $\Delta_r$. 
\end{definition}
It is worth noting that the definition of $\tc^\D_r(X)$ generalizes an earlier construction from \cite{FGLO17} which treats the case of $r=2$. The following statement and its proof occur as Lemma 4.2 and Corollary 4.3 in \cite{FOSequ}.
\begin{theorem}\label{rthm4}
Let $K$ be a connected finite aspherical cell complex of type $K(\pi,1)$  and let
$q: \widehat{K^r}\to K^r$ be the connected covering space
corresponding to $\Delta_r\subset \pi^r $. Then 
$$\tc^\D_r(K)=\tc_r(\pi)=\secat(q:\widehat{K^r}\to K^r).$$
\end{theorem}

One way to prove the second equality of Theorem \ref{rthm4} is the following: 

Let $X$ be a cell complex with $\pi_1(X)=\pi$ and let $\widetilde{X}$ be its universal cover. We can realize $\widehat{X^r}$, the covering space of $X^r$ that is associated with $\Delta_r\subset \pi^r$ as 
$$\widehat{X^r} = \widetilde{X}^r/\Delta_r,$$
i.e. as the orbit space of the $\Delta_r$-action on $\widetilde{X}^r$ obtained by restricting the $\pi^r$-action that is given as the $r$-fold product of the  $\pi$-action on $\widetilde{X}$ by deck transformations. Let $\rho: \widetilde{X}^r \to \widehat{X^r}$ denote the corresponding orbit space projection. 
Then there is a well-defined continuous map 
$$\phi_X: PX \to \widehat{X^r}, \quad \phi(\gamma) = \rho(\tilde{\gamma}(0),\tilde{\gamma}(\tfrac{1}{r-1}),\dots,\tilde\gamma(\tfrac{r-2}{r-1}),\tilde\gamma(1)),$$
where $\tilde\gamma$ denotes a lift of $\gamma$ to $\widetilde{X}$. (This map was first studied in the case of $r=2$ in \cite[Theorem 4.1]{FTY}.)
Then the following diagram commutes:
\begin{equation}
\label{EqFTYdiag}	
\begin{tikzcd}
	PX \ar[dr, swap, "p"] \ar[rr, "\phi_X"] & & \widehat{X^r} \ar[dl, "q"] \\
	& X^r &
\end{tikzcd}
\end{equation}
If $X=K$ is aspherical, then $\phi_K$ is a fibre homotopy equivalence since both $PK$ and $\widehat{K^r}$ are of type $K(\pi,1)$ and both of their  fundamental groups map isomorphically onto
$\Delta_r \subset \pi^r$.  Since $\phi_K$ commutes with the two fibrations, it follows from Dold's theorem, see \cite[Section 7.5]{MayConcise} that $\phi_K$ is a fibre homotopy equivalence, from which we derive the assertion of Theorem \ref{rthm4}. \medskip 

We derive an interesting observation:

\begin{proposition}
\label{PropCanonAspher}
Let $r \in \NN$ with $r \geq 2$, let $\pi$ be a geometrically finite group and let $K$ be a finite cell complex of type $K(\pi,1)$. Let $q:\widehat{K^r} \to K^r$ be a covering that is associated with the diagonal subgroup $\Delta_r \subset \pi^r$. 
\begin{enumerate}[a)]
	\item Up to identifications of coefficient modules, the canonical class $\vv_r \in H^1(K;I_r)$ is the homological obstruction to the existence of a continuous section of $q$. 
	\item It holds that 
$$	\TC^{\D}_r(K) \geq \height(\vv_r).$$
\end{enumerate}
\end{proposition}
\begin{proof}
\begin{enumerate}[a)]
	\item This follows from Proposition \ref{PropObstr} and the observations preceding this proposition.
	\item By construction, $\widehat{X^r}$ is aspherical with $\pi_1(\widehat{X^r}) = \Delta_r$. Thus, by Theorem \ref{rthm4},
	$$\TC^{\D}_r(K) = \secat(q:\widehat{K^r} \to K^r) = \secat(\Delta_r \hookrightarrow \pi^r).$$
 Using a), the claim is then shown along the same lines as Corollary \ref{CorSeqTCCanonical}.a).
\end{enumerate}
\end{proof}


\subsection{Beyond the aspherical case} By Theorem \ref{rthm4}, for a finite aspherical cell complex $K$, it holds that $\TC_r(K) = \TC_r^{\D}(K)$. This equality does not need to hold for arbitrary finite CW complexes. Counterexamples are provided by simply connected CW complexes for whom it follows straight from the definition that $\TC^{\D}_r(X)=0$ for all $r \geq 2$ which is certainly not true for $\TC_r(X)$ unless $X$ is contractible. However, there are more general results on relations between $\TC_r$ and $\TC^{\D}_r$.

The following theorem is proven by straightforward generalizations of the corresponding results for $r=2$, which occur as Propositions 2.2 and 2.4 in \cite{FGLOupper}.
\begin{theorem}
\label{ThmTCDgeneral}
Let $X$ be a connected locally finite cell complex with $\pi_1(X)=\pi$ and let $r \in \NN$ with $r \geq 2$. 
\begin{enumerate}[a)]
	\item Then $\TC_r(X)  \geq \TC^{\D}_r(X)$.
	\item Let $q^X:\widehat{X}^r \to X^r$ be the covering of $X^r$ that is associated with $\Delta_r \subset \pi^r$. Then 
	$$ \TC^{\D}_r(X) = \secat(q^X:\widehat{X}^r \to X^r).$$
\end{enumerate}
\end{theorem}
We can use this observation to establish a lower bound for sequential topological complexity from our considerations of the aspherical case. The following assertion was shown in \cite[Lemma 2.9]{FGLOupper} for the case of $r=2$, but again the proof generalizes straightforwardly to the sequential setting. 

\begin{lemma}
\label{LemmaTCDconn}
Let $r \in \NN$ with $r \geq 2$, let $X$ be a connected cell complex and put  $\pi :=\pi_1(X)$. Let $k \in \NN$, such that the universal cover of $X$ is $(k-1)$-connected. If $\cd(\pi) \leq k$, then $\TC_r^{\D}(X) = \TC_r(\pi)$.
\end{lemma}

\begin{corollary}
\label{CorSeqTClowerNonAsph}
Let $\pi$ be a geometrically finite group and $X$ be a connected locally finite cell complex with $\pi_1(X)=\pi$. Let $k \in \NN$, such that the universal cover of $X$ is $(k-1)$-connected. If $\cd(\pi) \leq k$, then
$$\TC_r(X) \geq r \cdot \cd(\pi) - k(\pi), $$
where $k(\pi) = \max \{\cd(C(g)) \ | \ g \in \pi\setminus\{1\}\}$.
\end{corollary}
\begin{proof}
	Combining Theorem \ref{ThmTCDgeneral}.a) with Lemma \ref{LemmaTCDconn} shows that $\TC_r(X) \geq \TC_r(\pi)$ in this setting. The claim it then an immediate consequence of Theorem \ref{TheoremLowerSeqTC}.
\end{proof}

We want to show next also for a finite cell complex $X$ that is not aspherical, the height of its $r$-th canonical class provides a lower bound for $\TC_r(X)$.  Put $G:= \pi^r$ and consider the universal covering $\widetilde{X}^r\to X^r$ as a principal $G$-fibration. It is shown in \cite{FOSequ} that the associated fibration
$$p^X:\widetilde{X^r}\times_G \pi^{r-1}\to X^r, $$
defined with respect to the $G$-action on $\pi^{r-1}$ from \eqref{raction} and viewing $\pi^{r-1}$ as a discrete group, coincides with the covering $q^X: \widehat{X}^r\to X^r$ from Theorem \ref{ThmTCDgeneral}.b). Thus, by Theorem \ref{ThmTCDgeneral}.b),
 $$\tc^\D_r(X)= \secat\big(f_r:\widetilde{X}^r\times_G \pi^{r-1}\to X^r\big).$$
It is a result of A. Schwarz, see \cite[Theorem 3]{Schwarz66}, that the sectional category of a fibration $p: E\to B$ equals the smallest integer $k$ such that the fiberwise join $p\ast p\ast \dots \ast p$ of $k+1$ copies of $p: E\to B$ admits a continuous section.
One checks that the join of $k+1$ copies of the fibration $p^X$ is given by 
$$q^X_k \colon \widetilde X^r \times_G E_k(\pi^{r-1}) \to X^r,$$
where $ E_k(\pi^{r-1})= \pi^{r-1} * \pi^{r-1} * \dots * \pi^{r-1}$ denotes the $(k+1)$-fold join of $\pi^{r-1}$ and where the left $G$-action on $E_k(\pi^{r-1})$ is induced by the one on $\pi^{r-1}$.  Thus, we obtain that 
\begin{equation}
\label{EqTCDinf}
\TC^{\D}_r(X) = \inf \{k \in \NN \ | \ q_k \colon \widetilde X^r \times_G E_k(\pi^{r-1}) \to X^r \  \text{admits a continuous section}\}.
\end{equation}

 Let $K$ be a space of type $K(\pi,1)$ and let $f_X\colon X \to K$ be the classifying map of the universal cover of $X$. Then $f_X^r:X^r \to K^r$ is the classifying map of the universal cover of $X^r$ for any $r \geq 2$. 

Consider the coverings $q^X:\widehat{X^r}\to X^r$ and $q:\widehat{K^r} \to K^r$ corresponding to the subgroup $\Delta_r \subset \pi^r$. Then the following is a pullback diagram:
$$\begin{tikzcd}
	\widetilde{X}^r\times_G \pi^{r-1} \ar[r] \ar[d, swap, "p^X"] & \widetilde{K}^r\times_G \pi^{r-1} \ar[d, "p^K"] \\
	X^r \ar[r, "f_X^r"] & K^r.
\end{tikzcd}$$
Thus, all of the diagrams associated to fibrewise joins of the two fibrations with themselves are pullbacks as well. That is, the following is a pullback diagram for any $k \in \NN$:
\begin{equation}
\label{EqPullbackDiag}
\begin{tikzcd}
	{\widetilde X}^r\times_G E_k(G/\Delta) \ar[r] \ar[d, swap, "q_k^X"] & {\widetilde K}^r \times_G E_k(G/\Delta) \ar[d, "q^K_k"] \\
	X^r \ar[r, "f_X^r"] & K^r.
\end{tikzcd}
\end{equation}
 
Combining this with the previous observations shows us how we can generalize properties of canonical classes to cell complexes that are not necessarily aspherical. 

\begin{theorem}
\label{TheoremTCDlower}
	Let $r \in \NN$ with $r \geq 2$ and let $X$ be a connected finite cell complex. Let $\vv_r \in H^1(X;I_r)$ be the $r$-th canonical class of $X$. Then 
	$$\TC^{\D}_r(X) \geq \height(\vv_r).$$
\end{theorem}
\begin{proof}
	Let $\pi=\pi_1(X)$ and let $k \in \NN$ with $\vv_r^k \neq 0$. By Corollary \ref{CorCanonicalPull}, $\vv_r^k = ((f^r)^*(\vv_{r,\pi}))^k = (f^r)^*(\vv_{r,\pi}^k)$, which shows that $\vv^k_{r,\pi} \neq 0$ as well.  By Proposition \ref{PropCanonAspher}.a) and \cite[Proposition 12]{Schwarz66},  $\vv^k_{r,\pi}$ is the homological obstruction to the existence of a continuous section of the fibration $$q_{k-1}^K: \widetilde{K}^r \times_G E_{k-1}(\pi^{r-1}) \to K^r. $$ Since homological obstructions are preserved under pullbacks, we derive from the pullback diagram \eqref{EqPullbackDiag} that $q^X_{k-1}:\widetilde{X}^r \times_G E_{k-1}(\pi^{r-1}) \to X^r$ does not admit a continuous section as well. Thus, it follows from \eqref{EqTCDinf} that $\TC^{\D}_r(X) \geq k$, which yields the claim.
\end{proof}

\begin{remark}
\label{RemarkTCrheight}
Evidently, combining Theorem \ref{ThmTCDgeneral} with Theorem \ref{TheoremTCDlower}	yields that
$$\TC_r(X) \geq \height(\vv_r)$$
for all finite cell complexes $X$ and all $r \geq 2$. This is implicitly shown in \cite{Costa} in the case of $r=2$. However, even in the case of $r=2$ it has hitherto not been shown that the height of the canonical class actually yields a lower bound not only for $\TC(X)$, but also for $\TC^{\D}(X)$.
\end{remark}

\subsection{Coincidence of $\TC_r$ and $\TC^{\D}_r$} We want to conclude our considerations by showing that $\TC_r$ and $\TC_r^{\D}$ coincide for some non-aspherical cell complexes under an additional condition that weakens the asphericity assumption. We view our result as an analogue of a result of A. Dranishnikov from \cite{DranishLens} in which the Lusternik-Schnirelmann category $\cat(X)$ and the $1$-category $\cat_1(X)$ are related in a similar way. The proof of Dranishnikov's result relied on a result on $k$-equivalences of joins of CW complexes.

We recall that a continuous map $f \colon X \to Y$ is a $k$-equivalence if $f_\#\colon \pi_j(X) \to \pi_j(Y)$ is an
isomorphism for $j < k$ and is a surjection for $j=k$. This is equivalent to saying that the relative homotopy groups
$\pi_j(Y,X)$ vanish for $j \leq k$.

\begin{proposition}[{\cite[Proposition 5.7]{DKR}}]
\label{prop:joinequiv}
Let $s \in \NN$ and let $X_1,\dots,X_s,Y_1,\dots,Y_s$ be CW complexes. Suppose that $f_j\colon X_j \to Y_j$ is an $n_j$-equivalence for $j\in \{1,2,\ldots,s\}$. Then the induced map on joins
$$(f_1 * \cdots * f_s)\colon X_1 * \cdots * X_s \to Y_1 * \cdots * Y_s $$
is a $(\min\{n_j \ | \ j \in \{1,2,\dots,s\}\} + s -1)$-equivalence.
\end{proposition}

In general, consider an $n$-equivalence $f\colon E \to E'$ of two total spaces of fibrations of cell complexes over the same base
$B$. By comparing the long exact homotopy sequences of the fibrations, we see that $f$ induces an
$n$-equivalence of fibres $\tilde f\colon F \to F'$. Let $f_k\colon \Gamma_k(E) \to \Gamma_k(E')$
be the induced map on fibrewise joins with restriction to fibres given by $\tilde f_k\colon *^{k+1} F \to
*^{k+1} F'$. By Proposition \ref{prop:joinequiv}, $\tilde f_k$ is an $(n+k+1-1)=(n+k)$-equivalence. 

This in turn implies that $f_k$ itself is an $(n+k)$-equivalence by again comparing long exact homotopy sequences. In fact, if $f\colon E \to E'$
is a map of fibrations over $B$ and the induced map of fibres is an $n$-equivalence, then so is $f$. With this in mind, we will tackle the proof of the following result.

\begin{theorem}\label{thm:TCDequalTC}
Let $k,r \in \NN$ with $r \geq 2$ and suppose $X$ is an $n$-dimensional CW complex whose universal cover $\widetilde X$ is $(rn-k)$-connected. 
\begin{enumerate}[a)]
\item If $\TC_r^{\D}(X)\geq k$, then $\TC_r(X)=\TC_r^{\D}(X)$.
	\item If $\TC_r^{\D}(X) \leq k$, then $\TC_r(X) \leq k$ as well.
\end{enumerate}
\end{theorem}

\begin{proof}
\begin{enumerate}[a)]
	\item We first note that if $k \leq (r-1)n$, then the assumption that $\widetilde X$ is $(rn-k)$-connected implies that
$\widetilde X$ is contractible, since it is $n$-dimensional and simply connected. Hence, $X$ is of type $K(\pi,1)$ and the claim follows from Theorem \ref{rthm4}.

We therefore consider the case of $k > (r-1)n$. Suppose $\TC^{\D}_r(X)=m \geq k$ and consider the diagram \eqref{EqFTYdiag}. Since $\widetilde X$ is $(rn-k)$-connected, we see that map induced by $\phi_X:PX \to \widehat{X^r}$ on fibres $(\Omega X)^{r-1} \to \pi^{r-1}$ is an $(rn-k)$-equivalence. As explained above, this yields that $\phi_X$ is an $(rn-k)$-equivalence as well. Therefore, by the discussion above, the induced map on fibrewise joins
$$\phi_m \colon \Gamma_m(p) \to \Gamma_m(q)$$
is an $(rn-k+m=rn+(m-k))$-equivalence with $m-k\geq 0$. By Theorem \ref{ThmTCDgeneral}.b), it holds that $\secat(q)=\TC_r^{\D}(X)=m$, so $q_m$ admits a continuous section $X^r \to \Gamma_m(q)$. The obstructions to lifting this section along $\phi_X$ to $\Gamma_m(p)$ lie in the groups $H^{i}(X^r;\pi_{i-1}(\mathcal{F}))$, where $\mathcal{F}$ is the homotopy fibre of $\phi_m$. Because $\phi_m$ is an $(rn+(m-k))$-equivalence, we know that $\pi_{i-1}(\mathcal{F})=\pi_i(\Gamma_m(q),\Gamma_m(p))=\{0\}$ for
$i \leq rn+(m-k)$. Since $X$ is assumed to be $n$-dimensional, $X^r$ is $rn$-dimensional, so all of the obstructions 
vanish and there is a section $X^r \to \Gamma_m(p)$. Hence, we have $\TC_r(X) \leq m = \TC^{\D}_r(X)$. Together with Theorem \ref{ThmTCDgeneral}.a), this shows the desired equality. 

\item If $\TC_r^{\D}(X) = m \leq k$, then the naturality of fibrewise joins implies that there is a section of $\Gamma_k(q) \to X^r$ 
induced by the section of $\Gamma_m(q) \to X^r$ and the natural maps $\Gamma_m(q) \to \Gamma_k(q)$ over $X^r$.
In this case we may apply the argument above verbatim to obtain a section of $\Gamma_k(p) \to X^r$, showing that $\TC_r(X) \leq k$.
\end{enumerate}
\end{proof}

The previous discussion has several interesting consequences.

\begin{corollary}
\label{CorTCDmax}
Suppose $X$ is a connected $n$-dimensional CW complex and $\TC_r(X) =rn$. Then 
$\TC_r^{\D}(X) = rn$ as well.
\end{corollary}

\begin{proof}
The universal cover $\widetilde X$ is $1$-connected, so take $k=rn-1$ in Theorem \ref{thm:TCDequalTC}. 
By part b) of that theorem, if $\TC_r^{\D}(X) \leq rn-1$, then the same would be true for $\TC_r(X)$, so
$\TC_r^{\D}(X) > rn-1$. Hence, $\TC_r^{\D}(X)=rn$.
\end{proof}

\begin{corollary}
	Suppose $X$ is a connected $n$-dimensional CW complex and that the universal cover $\widetilde X$ is $(rn-k)$-connected with
	$\TC_r^{\D}(X) \geq k$ and $\cd(\pi) \leq rn - k +1$. Then $\TC_r(X) = \TC_r(\pi)$.
\end{corollary}

\begin{proof}
	This follows from combining Theorem \ref{thm:TCDequalTC}, which yields that $\TC_r(X) = \TC_r^{\D}(X)$ in this case, with Lemma \ref{LemmaTCDconn}.\end{proof}
	
The following result gives a cohomological characterization of the maximality of sequential topological complexities. It generalizes the Costa-Farber theorem \cite[Theorem 7]{Costa} dealing with the case of $r=2$.

\begin{theorem}
\label{TheoremMaximality}
Let $n,r \in \NN$ with $r \geq 2$ and let $X$ be an $n$-dimensional CW complex. It holds that $\TC_r(X) = r  n$ if and only if $\vv_r^{r\cdot n} \neq 0$, where $\vv_r$ denotes the canonical class of $X$.
\end{theorem}
\begin{proof}
	If $\vv_r^{r  n}\neq 0$, then we derive from Remark \ref{RemarkTCrheight} that $\TC_r(X) \geq r  n$. The converse inequality follows from the standard upper bound of Theorem \ref{TheoremSecatProperties}.a), yielding that $\TC_r(X) \leq \dim (X^r) = rn$.
	
	Conversely, suppose that $\TC_r(X) = r  n$. Then Corollary \ref{CorTCDmax} yields that $\TC^{\D}_r(X)=rn$, so that by \eqref{EqTCDinf}, the fibrewise join $q_{rn-1}$ does not admit a continuous section. By Proposition \ref{PropObstr}.a) and \cite[Proposition 12]{Schwarz66}, the class $\vv_r^{rn}$ is the primary obstruction to a continuous section of $q_{rn-1}$. As previously discussed, the obstruction classes for the existence of continuous sections of $q_{rn-1}$ lie in the groups $H^i(X^r; \pi_{i-1}(E_{rn-1}(\pi)))$, for $1 \leq i \leq rn$. Recall, however, that the space $E_{rn-1}(\pi)$ is the $(rn)$-fold join $$ E_{rn-1}(\pi) = \pi^{*rn} $$ and consequently is a $(rn-2)$-connected space, which implies that all but the primary obstruction necessarily vanish. Therefore, $\vv_r^{r  n} \neq 0$.
\end{proof}

\begin{corollary}
\label{CorAbelian}
	Let $X$ be a connected $n$-dimensional finite CW complex, where $n \in \NN$, whose fundamental group is free abelian of rank at most $n$. Then 
	$$ \TC_r(X) < rn \qquad \forall r \geq 2. $$
\end{corollary}
\begin{proof}
By the hypothesis on $\pi_1(X)$ and Remark \ref{RemarkTCr}.(3), it holds that $\TC_r(\pi_1(X))<rn$, so that $\vv_{r,\pi_1(X)}^{rn}= 0$ by Theorem \ref{TheoremMaximality}. We derive from Corollary \ref{CorCanonicalPull}, that $\vv_r^{rn} = 0$ as well, where $\vv_r$ denotes the $r$-th canonical class of $X$. Therefore, the claim follows from Theorem \ref{TheoremMaximality}. 
\end{proof}

\begin{remark}
The non-maximality of topological complexity for closed manifolds with abelian fundamental groups was investigated by D. Cohen and L. Vandembroucq in \cite{CohVan}. While most of their results involve fundamental groups with torsion, see \cite[Theorem 1.2.(2a)]{CohVan} for a result similar to Corollary \ref{CorAbelian}.
\end{remark}

\bibliography{bibliography}{}
\bibliographystyle{plain}

\end{document}